\documentclass[10pt,reqno]{amsart}
 \newtheorem{thm}{Theorem}[section]
 \newtheorem{cor}[thm]{Corollary}
 \newtheorem{lem}[thm]{Lemma}
 \newtheorem{prop}[thm]{Proposition}
 \theoremstyle{definition}
 
 \theoremstyle{remark}
 \newtheorem{rem}[thm]{Remark}
 \theoremstyle{remark}
 \newtheorem{rems}[thm]{Remarks}

 \numberwithin{equation}{section}

\newcommand{\nn}{\mathbb N}
\newcommand{\ZZ}{\mathbb Z}

\newcommand{\C}{\mathbb C}
\newcommand{\R}{\mathbb{R}}
\newcommand{\BB}{\mathbb{B}}
\newcommand{\EE}{\mathbb E}
\newcommand{\FF}{\mathbb F}

\newcommand{\cR}{{\mathcal R}}

\newcommand{\cH}{{\mathcal H}}
\newcommand{\q}{q}

\def\ba{\bar}

\def\eps{\varepsilon}

\newcommand{\N}{\mathbb N}
\newcommand{\Li}{\mathcal L}

\DeclareMathSymbol{\complement}{\mathord}{AMSa}{"7B}

\def\vv<#1>{\langle #1\rangle}
\def\Vv<#1>{\bigl\langle #1\bigr\rangle}

\begin{document}
\title[Two-Phase Navier-Stokes Equations]
{Analytic solutions for the two-phase Navier-Stokes equations with surface tension and gravity}

\author[J. Pr\"uss]{Jan Pr\"uss}
\address{Institut f\"ur Mathematik  \\
         Martin-Luther-Universit\"at Halle-Witten\-berg\\
         Theodor-Lieser-Str.~5\\
         D-60120 Halle, Germany}
\email{jan.pruess@mathematik.uni-halle.de}

\author[G. Simonett]{Gieri Simonett}
\address{Department of Mathematics\\
           Vanderbilt University
           Nashville, TN}
\email{gieri.simonett@vanderbilt.edu}

\thanks{The research of GS was partially
supported by NSF, Grant DMS-0600870.}



\dedicatory{Dedicated to Herbert Amann on the occasion of his 70th birthday}

\begin{abstract}
We consider the motion of two superposed immiscible, viscous, incompressible, capillary fluids that are separated by a sharp interface which needs to be 
determined as part of the problem. 
Allowing for gravity to act on the fluids, we
prove local well-posedness of the problem.
In particular, we obtain well-posedness for the case where
the heavy fluid lies on top of the light one, that is, for the case
where the Rayleigh-Taylor instability is present.
Additionally we show that solutions become real analytic instantaneously.
\end{abstract}

\maketitle
\section{Introduction and Main Results}
We consider a free boundary problem describing the
motion of two immiscible, viscous, incompressible capillary fluids,
{\em fluid$_1$} and {\em fluid$_2$},
occupying the regions
\begin{equation*}
\Omega_i(t)=\{(x,y)\in\R^n\times\R: (-1)^i(y-h(t,x))>0,\ t\ge 0\},
\quad i=1,2.
\end{equation*}
The fluids, thus, are separated by the interface
\begin{equation*}
\Gamma(t):=\{(x,y)\in\R^n\times\R: y=h(t,x):x\in\R^n,\ t\ge 0\},
\end{equation*}
called the free boundary,
which needs to be determined as part of the problem.
The motion of the fluids is governed by the incompressible
Navier-Stokes equations where surface tension
on the free boundary is included.
In addition, we also allow for gravity to act on the fluids.
The governing equations then are given by the system
\begin{equation}
\label{NS-2}
\left\{
\begin{aligned}
 \rho\big(\partial_tu+(u|\nabla)u\big)-\mu\Delta u+\nabla \q
                   &=  0\ &\hbox{in}\quad &\Omega(t)\\
      {\rm div}\,u & = 0\ &\hbox{in}\quad &\Omega(t) \\
     -{[\![S(u,\q)\nu]\!]}& = \sigma\kappa \nu +[\![\rho]\!]\gamma_a y&\hbox{on}\quad &\Gamma(t)\\
       {[\![u]\!]} & = 0\ &\ \hbox{on}\quad &\Gamma(t) \\
                  V& = (u|\nu) &\ \hbox{on}\quad &\Gamma(t)\\
               u(0)& = u_0&\ \hbox{in}\quad &\Omega_{0}\\
          \Gamma(0)& = \Gamma_0\,.\\
\end{aligned}
\right.
\end{equation}
Here $\rho$ and $\mu$ are given by
\begin{equation*}
\rho=\rho_1\chi_{\Omega_1(t)}+ \rho_2\chi_{\Omega_2(t)},
\quad
\mu=\mu_1\chi_{\Omega_1(t)}+\mu_2\chi_{\Omega_2(t)},
\end{equation*}
with $\chi$ the indicator function,
where the constants $\rho_i$ and $\mu_i$ denote the densities
and viscosities of the respective fluids.
The constant $\sigma>0$ denotes the surface tension,
and $\gamma_a$ is the acceleration of gravity.
Moreover,
$S(u,q)$ is the stress tensor defined by
\begin{equation*}
S(u,\q)=\mu_i\big(\nabla u+(\nabla u)^{\sf T}\big)-qI
\quad \text{in}\quad \Omega_i(t),
\end{equation*}
where $q=\tilde q+\rho\gamma_ay$ denotes the modified pressure
incorporating the potential of the gravity force,
and
$$
[\![v]\!]=(v_{|_{\Omega_2(t)}}-v_{|_{\Omega_1(t)}}\big)|_{\Gamma(t)}
$$
denotes the jump of the quantity $v$, defined on the respective
domains $\Omega_i(t)$, across the interface $\Gamma(t)$.
Finally,
$\kappa=\kappa(t,\cdot)$ is the mean curvature of the free boundary $\Gamma(t)$,
$\nu=\nu(t,\cdot)$ is the unit normal field on
$\Gamma(t)$, and $V=V(t,\cdot)$ is the normal velocity of $\Gamma(t)$.
Here we use the convention that
$\nu(t,\cdot)$ points from $\Omega_1(t)$ into $\Omega_2(t)$, and
that $\kappa(x,t)$ is negative when $\Omega_1(t)$ is convex
in a neighborhood of $x\in\Gamma(t)$.
\smallskip\\
Given are the initial position $\Gamma_{0}=\text{graph}\,(h_0)$
of the interface, and the initial velocity
$$
u_0:\Omega_0\to \R^{n+1},\quad \Omega_0:=\Omega_1(0)\cup\Omega_2(0).
$$
The unknowns are the velocity field
$u(t,\cdot):\Omega(t)\to \R^{n+1}$,
the pressure field $\q(t,\cdot):\Omega(t)\to \R$,
and the free boundary $\Gamma(t)$,
where $\Omega(t):=\Omega_1(t)\cup\Omega_2(t)$.
\smallskip\\
Our main result shows that problem~\eqref{NS-2} admits
a unique local smooth solution, provided that 
$\|\nabla h_0\|_\infty:=\sup_{x\in\R^n}|\nabla h_0(x)|$ is sufficiently small.
\begin{thm}
\label{th:1.1}
Let $p>n+3$.
Then given $\beta>0$, there exists $\eta=\eta(\beta)>0$ such
that for all initial values
$$(u_0, h_0)\in W^{2-2/p}_p(\Omega_0,\R^{n+1})
         \times W^{3-2/p}_p(\R^n),
\quad [\![u_0]\!]=0,
$$
satisfying the compatibility conditions
$$
[\![\mu D(u_0)\nu_0-\mu(\nu_0|D(u_0)\nu_0)\nu_0]\!]=0,
\quad {\rm div} \; u_0=0\; \text{ on }\; \Omega_0,
$$
with $D(u_0):=(\nabla u_0+(\nabla u_0)^{\sf T})$,
and the smallness-boundedness condition
$$
\|\nabla h_0\|_\infty\le\eta,\qquad\|u_0\|_{\infty}\le\beta,
$$
there is $t_0=t_0(u_0,h_0)>0$
such that problem \eqref{NS-2}
admits a classical solution $(u,q,\Gamma)$ on $(0,t_0)$.
The solution is unique in the function class described in
Theorem~\ref{th:nonlinear}.
In addition, $\Gamma(t)$
is a graph over $\R^n$ given
by a function $h(t)$ and
${\mathcal M}=\bigcup_{t\in(0,t_0)}\big(\{t\}\times\Gamma(t)\big)$
is a real analytic manifold, and with
$${\mathcal O}:=\{(t,x,y):\; t\in(0,t_0),\; x\in\R^n,
y\neq h(t,x)\},$$
the function $(u,\q):\mathcal O\rightarrow\R^{n+2}$ is real analytic.
\end{thm}
\begin{rems}
\label{rem:1.2}
(a)
More precise statements for the transformed problem will be given in
Section 4. Due to the restriction $p>n+3$ we obtain
\begin{equation*}
h\in C(J;BU\!C^2(\R^n))\cap C^1(J;BU\!C^1(\R^n)),
\end{equation*}
where $J=[0,t_0]$. In particular, the normal of $\Omega_1(t)$, the
normal velocity of $\Gamma(t)$, and the mean curvature of $\Gamma(t)$
are well-defined and continuous, so that \eqref{NS-2} makes sense
pointwise. For $u$ we obtain
\begin{equation*}
u\in BU\!C(J\times\R^{n+1},\R^{n+1}),
\quad \nabla u\in BU\!C({\mathcal O},\R^{(n+1)^2}).
\end{equation*}
Also interesting is the fact that the surface pressure jump
is analytic on ${\mathcal M}$ as well.
\smallskip\\
(b)
It is possible to relax the assumption $p>n+3$. In fact, $p>(n+3)/2$
turns out to be sufficient.
In order to keep the arguments simple, we impose here
the stronger condition $p>n+3$.
\smallskip\\
(c)
It is well-known that the situation
where gravity is acting on two superposed immiscible fluids - with the heavier fluid lying above a fluid of lesser density - leads to an instability, the
 Rayleigh-Taylor instability. In this case, small disturbances
of the equilibrium situation $(u,h)=(0,0)$ can cause instabilities, where
the heavy fluid moves down under the influence of gravity, and the light
material is displaced upwards, leading to vortices.
Our results show that problem \eqref{NS-2} is also well-posed in this case,
provided $\|\nabla h_0\|_\infty$ is small enough,
yielding smooth solutions for a short time.
In the forthcoming publication \cite{PrSi09c} we will give a rigorous proof
showing that the equilibrium solution $(u,h)=(0,0)$ is $L_p$-unstable.
To the best of our knowledge these are the first rigorous results
concerning the Navier-Stokes equations subject to the Rayleigh-Taylor instability.
\smallskip\\
(d)
If $\gamma_a=0$
then it is shown in \cite{PrSi09a} that problem~\eqref{NS-2} admits a solution
with the same regularity properties on an arbitrary fixed time interval $[0,t_0]$,
provided that $\|u_0\|_{W^{2-2/p}_p(\Omega_0)}$ and $\|h_0\|_{W^{3-2/p}_p(\R^n)}$
are sufficiently small (depending on $t_0$).
\medskip\\
(e) We point out that in Theorem 1.1 we only need a smallness condition
on the sup-norm of $\nabla h_0$ (relative to the vertical component of the velocity).
In case of a more general geometry, this condition can always be achieved by a
judicious choice of a reference manifold.
\end{rems}
The motion of a layer of  viscous,
incompressible fluid in an ocean of infinite extent,
bounded below by a solid surface and above by a free surface which includes
the effects of surface tension and gravity
(in which case $\Omega_0$ is a strip, bounded above
by $\Gamma_0$ and below by a fixed surface $\Gamma_b$)
has been considered by
Allain~\cite{Al87},
Beale~\cite{Bea84}, Beale and Nishida~\cite{BeaNi84},
Tani~\cite{Ta96},  by Tani and Tanaka~\cite{TT95},
and by Shibata and Shimizu~\cite{SS09}.
If the initial state and the initial velocity are close
to equilibrium, global existence of solutions is proved
in \cite{Bea84} for $\sigma>0$,
and in \cite{TT95} for $\sigma\ge 0$,
and the asymptotic decay rate for $t\to\infty$
is studied in~\cite{BeaNi84}.
We also refer to \cite{BPS05}, where in addition the
presence of a surfactant
on the free boundary and in one of the bulk phases is considered.
\smallskip\\
In case that $\Omega_1(t)$ is a bounded domain, $\gamma_a=0$, and
$\Omega_2(t)=\emptyset$, one obtains
the {\em one-phase} Navier-Stokes equations with surface tension,
describing the motion of an isolated volume of fluid.
For an overview of the existing literature in this case we refer
to the recent publications \cite{PrSi09a, SS07, SS09, So03b}.
\smallskip\\
Results concerning the {\em two-phase problem}~\eqref{NS-2}
with $\gamma_a=0$ in the $3D$-case
are obtained in \cite{Deni91, Deni94, DS95,Tanaka93}.
In more detail, Densiova~\cite{Deni94}
establishes existence and uniqueness of solutions
(of the transformed problem in Lagrangian coordinates)
with $v\in W^{s,s/2}_2$
for $s\in (5/2,3)$
in case that one of the domains is bounded.
Tanaka \cite{Tanaka93} considers the two-phase Navier-Stokes equations
with thermo-capillary convection in bounded domains, and he obtains
existence and uniqueness of solutions
with $(v,\theta)\in W^{s,s/2}_2$
for $s\in (7/2,4)$,
with $\theta$ denoting the temperature.
\smallskip\\
In order to prove our main result we
transform problem \eqref{NS-2} into a problem on a fixed domain.
The transformation is expressed in terms of the
unknown height function $h$ describing the free boundary.
Our analysis proceeds with establishing maximal regularity results
for an associated linear problem.
relying on the powerful theory of maximal regularity,
in particular on the  $H^\infty$-calculus for sectorial
operators, the Dore-Venni theorem, and the Kalton-Weis theorem,
see for instance~\cite{Am95, DHP03, DoVe87, KW01, KuWe04, PrSi06, PrSo90}.

Based on the linear estimates we can solve the nonlinear problem
by the contraction mapping principle.
Analyticity of solutions is obtained as in
\cite{PrSi09a} by the implicit function theorem in conjunction
with a scaling argument, relying on an idea that goes back
to Angenent \cite{Ang90a, Ang90b} and Masuda \cite{Ma80};
see also \cite{ES96, ES03, EPS03b}.


The plan for this paper is as follows. Section 2 contains the
transformation of the problem to a half-space and the determination
of the proper underlying linear problem.
In Section 3 we analyze this linearization and prove the
crucial maximal regularity result in an $L_p$-setting. Section~4
is then devoted to the nonlinear problem and contains the proof
of our main result.
Finally we collect and prove in an appendix some of the technical results
used in order to estimate the nonlinear terms.

\section{The transformed problem}
The nonlinear problem \eqref{NS-2} can be transformed to a problem
on a fixed domain by means of the transformations
\begin{equation*}
\begin{split}
&v(t,x,y):=(u_1,\ldots,u_n)(t,x,y+h(t,x)), \\
&w(t,x,y):=u_{n+1}(t,x,y+h(t,x)), \\
&\pi(t,x,y):=q(t,x,y+h(t,x)),
\end{split}
\end{equation*}
where $t\in J=[0,a]$, $x\in \R^n$, $y\in\R$, $y\neq0$.
With a slight abuse of notation we will
in the sequel denote the transformed velocity again
by $u$, that is, we set $u=(v,w)$.
With this notation we obtain the transformed problem
\begin{equation}
\label{tfbns2}
\left\{
\begin{aligned}
\rho\partial_tu -\mu\Delta u+\nabla \pi&= F(u,\pi,h)
    &\ \hbox{in}\quad &\dot\R^{n+1}\\
{\rm div}\,u&= F_d(u,h)&\ \hbox{in}\quad &\dot\R^{n+1}\\
-[\![\mu\partial_y v]\!] -[\![\mu\nabla_{x}w]\!] &=G_v(u,[\![\pi]\!],h)
    &\ \hbox{on}\quad &\R^n\\
-2[\![\mu\partial_y w]\!] +[\![\pi]\!] -\sigma\Delta h-[\![\rho]\!]\gamma_a h &= G_w(u,h)
    &\ \hbox{on}\quad &\R^n\\
[\![u]\!] &=0 &\ \hbox{on}\quad &\R^n\\
\partial_th-\gamma w&=-(\gamma v|\nabla h) &\ \hbox{on}\quad &\R^n\\
u(0)=u_0,\; h(0)&=h_0, \\\end{aligned}
\right.
\end{equation}
for $t>0$, where $\dot\R^{n+1}=\{(x,y)\in\R^n\times\R: y\neq 0\}$.
\medskip\\
The nonlinear functions have been computed in \cite{PrSi09a} and are given by:
\begin{equation}
\label{2.2}
\begin{split}
F_v(v,w,\pi,h)&=\mu\{- 2(\nabla h|
\nabla_{x})\partial_yv +|\nabla h|^2
\partial^2_yv-\Delta h\partial_yv\}+\partial_y \pi\nabla h \\
&\quad
+\rho\{-(v|\nabla_{x})v
+(\nabla h|v)\partial_yv-
w\partial_yv\}
+ \rho\partial_th \partial_y v, \\
F_w(v,w,h)&=\mu\{- 2(\nabla h|
\nabla_{x})\partial_yw +|\nabla h|^2
\partial^2_yw -\Delta h\partial_yw\} \\
&\quad  +\rho\{-(v|\nabla_{x})w
+(\nabla h|v)\partial_yw- w\partial_yw\}
+ \rho\partial_th \partial_y w, \\
F_d(v,h)&= (\nabla h|\partial_y v)
\end{split}
\end{equation}
and
\begin{equation}
\label{2.3}
\begin{split}
G_v(v,w,[\![\pi]\!],h)&\!=\!
-[\![\mu(\nabla_{x}v+(\nabla_{x}v)^{\sf T})]\!]\nabla h
+|\nabla h|^2[\![\mu\partial_yv]\!]
+ (\nabla h|\,[\![\mu\partial_yv]\!])\nabla h\\
&\quad -[\![\mu\partial_y w]\!]\nabla h
 + \{[\![\pi]\!]-\sigma(\Delta h-G_\kappa(h))\}\nabla h, \\
\ \ G_w(v,w,h)&\!=\!
- (\nabla h|\, [\![\mu\nabla_{x}w]\!])
- (\nabla h|\, [\![\mu\partial_yv]\!])
+ |\nabla h|^2  [\![\mu\partial_yw]\!]
-\sigma G_\kappa(h)
\end{split}
\end{equation}
with
\begin{equation}
\label{2.4}
G_\kappa(h)=\frac{|\nabla h|^2\Delta h}
                {(1+\sqrt{1+|\nabla h|^2})
                \sqrt{1+|\nabla h|^2}}
   +\frac{(\nabla h| \nabla^2 h\nabla h)}
                {(1+|\nabla h|^2)^{3/2}},
\end{equation}
where $\nabla^2 h$ denotes the Hessian matrix of all second
order derivatives of $h$.
\medskip

Before studying solvability results for problem \eqref{tfbns2}
let us first introduce suitable function spaces.
Let $\Omega\subseteq\R^m$ be open and $X$ be
an arbitrary Banach space.
By $L_p(\Omega;X)$ and $H^s_p(\Omega;X)$,
for $1\le p\le\infty$, $s\in\R$, we denote the $X$-valued Lebegue and
the Bessel potential spaces of order $s$, respectively.
We will also frequently make use of the fractional
Sobolev-Slobodeckij
spaces $W^s_p(\Omega;X)$, $1\le p< \infty$,
$s\in\R\setminus\ZZ$, with norm
\begin{equation}
\label{Slobodeskii}
         \|g\|_{W^s_p(\Omega;X)}
         =\|g\|_{W^{[s]}_p(\Omega;X)}
         +\sum_{|\alpha|=[s]}
          \left(\int_\Omega\int_\Omega
          \frac{\|\partial^\alpha g(x)-\partial^\alpha g(y)\|^p_{X}}
          {|x-y|^{m+(s-[s])p}}\, dx\, dy\right)^{\!1/p}{\hskip-4mm},
\end{equation}
where $[s]$ denotes the largest integer smaller than $s$.
Let $a\in(0,\infty]$ and $J=[0,a]$.
We set
\[
         _0W^s_p(J;X):=\left\{
         \begin{array}{l}
         \{g\in W^s_p(J;X): g(0)=g'(0)=\ldots=g^{(k)}(0)=0\},\\[3mm]
         \mbox{if}\quad
         k+\frac1{p}<s<k+1+\frac1{p}, \ k\in\N\cup\{0\},\\[3mm]
         W^s_p(J;X),\quad \mbox{if}\quad s<\frac1{p}.
         \end{array}
         \right.
\]
The spaces $_0H^s_p(J;X)$ are defined analogously.
Here we remind that $H^k_p=W^k_p$ for $k\in\ZZ$ and $1<p<\infty$,
and that $W^s_p=B^s_{pp}$ for $s\in\R\setminus\ZZ$.
\\
For $\Omega\subset\R^m $ open and $1\le p<\infty$,
the { homogeneous Sobolev spaces} $\dot H^1_p(\Omega)$ of order $1$
are defined as
\begin{equation}
\label{dot-H-1}
\begin{split}
&\dot H^1_p(\Omega):=(\{g\in L_{1,\text{loc}}(\Omega):
\|\nabla g\|_{L_p(\Omega)}<\infty\},\|\cdot\|_{\dot H^1_p(\Omega)}) \\
&\|g\|_{\dot H^1_p(\Omega)}:=
\big(\sum\limits_{j=1}^m \|\partial_j g\|^p_{L_p(\Omega)}\big)^{\!1/p}.
\end{split}
\end{equation}
Then $\dot H^1_p(\Omega)$
is a Banach space, provided we factor out the constant functions
and equip the resulting space with the corresponding quotient norm,
see for instance~\cite[Lemma II.5.1]{Ga94}.
We will in the sequel always consider the quotient space topology
without change of notation.
In case that $\Omega$ is locally Lipschitz, it is known that
$\dot H^1_p(\Omega)\subset H^1_{p,\text{loc}}(\overline\Omega)$,
see \cite[Remark II.5.1]{Ga94}, and consequently,
any function in  $\dot H^1_p(\Omega)$ has a well-defined trace on
$\partial\Omega$.

For $s\in\R$ and $1<p<\infty$ we
also consider the {homogeneous Bessel-potential spaces} $\dot H^s_p(\R^n)$
of order $s$, defined by
\begin{equation}
\label{dot-H-s}
\begin{split}
&\dot H^s_p(\R^n):=(\{g\in {\mathcal S}^\prime(\R^n): \dot I^sg\in L_p(\R^n)\},
\|\cdot\|_{\dot H^s_p(\R^n)}),\\
& \|g\|_{\dot H^s_p(\R^n)}:=\|\dot I^sg\|_{L_p(\R^n)},
\end{split}
\end{equation}
where ${\mathcal S^\prime}(\R^n)$
denotes the space of all tempered distributions,
and $\dot I^s$ is the Riesz potential given by
$$
\dot I^s g:=(-\Delta)^{s/2}g:=
\mathcal F^{-1}(|\xi|^s \mathcal Fg),\quad g\in\mathcal S^\prime(\R^n).
$$
By factoring out all polynomials,
$\dot H^s_p(\R^n)$
becomes a Banach space with the natural quotient norm.
For $s\in\R\setminus\ZZ$,
the homogeneous Sobolev-Slobodeckij spaces $\dot W^s_p(\R^n)$
of fractional order
can be obtained  by real interpolation as
\begin{equation*}
\dot W^s_p(\R^n):=(\dot H^k_p(\R^n),\dot H^{k+1}_p(\R^n))_{s-k,p},
\quad k<s<k+1,
\end{equation*}
where $(\cdot,\cdot)_{\theta,p}$ is the real interpolation method.
It follows that
\begin{equation}
\label{I-s-isom}
\dot I^s\in
\text{Isom}(\dot H^{t+s}_p(\R^n),\dot H^t_p(\R^n))\cap
\text{Isom}(\dot W^{t+s}_p(\R^n),\dot W^t_p(\R^n)),
\quad s,t\in\R,
\end{equation}
with $\dot W^k_p=\dot H^k_p$ for $k\in\ZZ$.
We refer to \cite[Section 6.3]{BL76} and \cite[Section 5]{Tr83}
for more information on homogeneous functions spaces.
In particular, it follows from
parts (ii) and (iii) in \cite[Theorem 5.2.3.1]{Tr83}
that the definitions \eqref{dot-H-1} and \eqref{dot-H-s}
are consistent if $\Omega=\R^n$, $s=1$, and $1<p<\infty$.
We note in passing that
\begin{equation}
\label{norms-homogeneous}
          \left(\int_{\R^n}\int_{\R^n}
          \frac{|g(x)-g(y)|^p}
          {|x-y|^{n+sp}}\, dx\, dy\right)^{\!1/p}\!\!,
\ \
\left(\int_0^\infty t^{(1-s)p}\|\frac{d}{dt}P(t)g\|^p_{L_p(\R^n)}\frac{dt}{t}
\right)^{\!1/p}\!\!
\end{equation}
define equivalent norms on $\dot W^s_p(\R^n)$ for $0<s<1$,
where $P(\cdot)$ denotes the Poisson semigroup,
see \cite[Theorem 5.2.3.2 and Remark 5.2.3.4]{Tr83}.
Moreover,
\begin{equation}
\label{trace-homogeneous}
\gamma_{\pm}\in \Li(\dot W^1_p(\R^{n+1}_\pm), \dot W^{1-1/p}_p(\R^n)),
\end{equation}
where $\gamma_{\pm}$ denotes the trace operators,
see for instance \cite[Theorem II.8.2]{Ga94}.

\section{The Linearized Two-Phase Stokes Problem with Free boundary}
It turns out that, unfortunately, the
 nonlinear term $(\gamma v|\nabla h)$ occurring in
 \eqref{tfbns2} cannot be made small in the norm of $\FF_4(a)$,
 defined below in \eqref{FF},
by merely taking $\|\nabla h\|_\infty$ small.
This can, however, be achieved for the modified term
$(b-\gamma v|\nabla h)$, provided $b$ is properly chosen
so that $b(0)=\gamma v_0.$
As a consequence, we now need to consider the
modified linear problem
\begin{equation}
\label{linFBpert}
\left\{
\begin{aligned}
\rho\partial_tu -\mu\Delta u+\nabla \pi&=f
    &\ \hbox{in}\quad &\dot\R^{n+1}\\
{\rm div}\,u&= f_d&\ \hbox{in}\quad &\dot\R^{n+1}\\
-[\![\mu\partial_y v]\!] -[\![\mu\nabla_{x}w]\!]&=g_v
    &\ \hbox{on}\quad &\R^n\\
-2[\![\mu\partial_y w]\!] +[\![\pi]\!] &=g_w +\sigma\Delta h +[\![\rho]\!]\gamma_a h
    &\ \hbox{on}\quad &\R^n\\
[\![u]\!] &=0 &\ \hbox{on}\quad &\R^n\\
\partial_th-\gamma w +(b(t,x)|\nabla)h&=g_h &\ \hbox{on}\quad &\R^n\\
u(0)=u_0,\; h(0)&=h_0. &\\
 \end{aligned}
\right.
\end{equation}
Here we mention that the simpler case where
$b=0$ and $\gamma_a=0$ was studied
in \cite[Theorem 5.1]{PrSi09a}.
We obtain the following maximal regularity result.
\goodbreak
\begin{thm}
\label{th:b}
Let $p>n+3$ be fixed, and assume that $\rho_j$ and $\mu_j$ are positive
constants for $j=1,2$, and set $J=[0,a]$. Suppose
$$
\quad b_0\in\R^n,\quad
b_1\in W^{1-1/2p}_p(J;L_p(\R^n,\R^n))\cap L_p(J;W^{2-1/p}_p(\R^n,\R^n)),
$$
and set $b(\cdot)=b_0+b_1(\cdot)$. Then the  Stokes problem with free boundary \eqref{linFBpert} admits a unique solution $(u,\pi,h)$ with regularity
\begin{equation}
\label{reg}
\begin{split}
&u\in H^1_p(J;L_p(\R^{n+1},\R^{n+1}))
  \cap L_p(J;H^2_p(\dot{\R}^{n+1},\R^{n+1})), \\
&\pi\in L_p(J;\dot{H}^1_p(\dot{\R}^{n+1})), \\
&[\![\pi]\!]\in W^{1/2-1/2p}_p(J;L_p(\R^{n}))\cap L_p(J;W^{1-1/p}_p(\R^{n})),\\
& h\in W^{2-1/2p}_p(J;L_p(\R^{n}))\cap H^1_p(J;W^{2-1/p}_p(\R^n))
\cap L_p(J;W^{3-1/p}_p(\R^n))
\end{split}
\end{equation}
if and only if the data
$(f,f_d,g,g_h,u_0,h_0)$
satisfy the following regularity and compatibility conditions:
\begin{itemize}
\item[(a)]
$f\in L_p(J;L_p(\R^{n+1},\R^{n+1}))$,
\vspace{1mm}
\item[(b)]
$f_d\in H^1_p(J; \dot{H}^{-1}_p(\R^{n+1}))\cap L_p(J; H^1_p(\dot{\R}^{n+1}))$,
\vspace{1mm}
\item[(c)]
$g=(g_v,g_w)\in W^{1/2-1/2p}_p(J;L_p(\R^{n},\R^{n+1}))
\cap L_p(J;W^{1-1/p}_p(\R^{n},\R^{n+1}))$,
\vspace{1mm}
\item[(d)]
$g_h\in W^{1-1/2p}_p(J;L_p(\R^{n}))\cap L_p(J;W^{2-1/p}_p(\R^{n}))$,
\vspace{1mm}
\item[(e)]
$u_0\in W^{2-2/p}_p(\dot{\R}^{n+1},\R^{n+1})$, $h_0\in W^{3-2/p}_p(\R^n)$,
\vspace{1mm}
\item[(f)]
${\rm div}\, u_0=f_d(0)$ in $\,\dot\R^{n+1}$ and $[\![u_0]\!]=0$
on $\,\R^n$ if $p>3/2$,
\vspace{1mm}
\item[(g)]
$-[\![\mu\partial_y v_0]\!] -[\![\mu\nabla_{x}w_0]\!] ={g_v}(0)$ on
$\,\R^n$ if $p>3$.
\end{itemize}
The solution map $[(f,f_d,g,g_h, u_0,h_0)\mapsto (u,\pi,h)]$ is continuous between the corresponding spaces.

If $b_1\equiv0$ then the result
is true for all
$p\in(1,\infty)$, $p\neq3/2,3$.
\end{thm}

\begin{proof}
(i) Since $\FF_4(a)$, defined by
$$\FF_4(a):=W^{1-1/2p}_p(J;L_p(\R^n))\cap L_p(J;W^{2-1/p}_p(\R^n)),$$
is a multiplication algebra for $p>n+3$,
the operator $[h\mapsto (b|\nabla)h]$ maps the space
$$\EE_4(a):=W^{2-1/2p}_p(J;L_p(\R^{n}))\cap H^1_p(J;W^{2-1/p}_p(\R^n))
\cap L_p(J;W^{3-1/p}_p(\R^n))$$
continuously into $\FF_4(a)$
with bound $|b_0|+ C_a \|b_1\|_{\FF_4(a)}$,
see Lemma \ref{le:3}(a).
\smallskip\\
As in the proof of \cite[Theorem 5.1]{PrSi09a}
it suffices to consider the reduced problem
\begin{equation}
\label{linFB0}
\left\{
\begin{aligned}
\rho\partial_tu -\mu\Delta u+\nabla \pi&=0
    &\ \hbox{in}\quad &\dot\R^{n+1}\\
{\rm div}\,u&= 0&\ \hbox{in}\quad &\dot\R^{n+1}\\
-[\![\mu\partial_y v]\!] -[\![\mu\nabla_{x}w]\!]&=0
    &\ \hbox{on}\quad &\R^n\\
-2[\![\mu\partial_y w]\!] +[\![\pi]\!]  &=\sigma\Delta h +[\![\rho]\!]\gamma_a h
    &\ \hbox{on}\quad &\R^n\\
[\![u]\!] &=0 &\ \hbox{on}\quad &\R^n\\
\partial_th-\gamma w+(b(t,x)|\nabla)h&={\tilde g}_h &\ \hbox{on}\quad &\R^n\\
u(0)=0,\; h(0)&=0, \\
\end{aligned}
\right.
\end{equation}
where the function $\tilde g_h\in {_0}\FF_4(a)$ is defined in a similar way
as in formula (5.5) in \cite{PrSi09a}.
This can be accomplished by choosing
$h_1:=h_{1,b}\in\EE_4(a)$
such that
\begin{equation*}
h_1(0)=h_0,\quad
\partial_t h_{1}(0)=g_h(0)+\gamma w_0-(b(0)|\nabla h_0),
\end{equation*}
and then setting
${\tilde g}_h:={\tilde g}_{h,b}:=g_h+\gamma w_1-(b|\nabla h_{1})-\partial_t h_{1},
$
where $w_1$ has the same meaning as in step (i) of the proof
of \cite[Theorem 5.1]{PrSi09a}.
\smallskip\\
(ii) We first consider the reduced problem~\eqref{linFB0}
for the case where $b\equiv b_0$ is constant.
The corresponding boundary symbol
$s_{b_0}(\lambda,\xi)$ is given by
\begin{equation}
s_{b_0}(\lambda,\xi)=\lambda+\big(\sigma |\xi|-[\![\rho]\!]\gamma_a/|\xi|\big)k(z)
+i(b_0|\xi),
 \end{equation}
where we use the same notation as in the proof of \cite[Theorem 5.1]{PrSi09a}.
Here we remind that $k$ has the following properties:
$k$ is holomorphic in $\C\setminus\R_{-}$ and
\begin{equation}
\label{k-properties}
k(0)=\frac{1}{2(\mu_1+\mu_2)},\quad zk(z)\to \frac{1}{\rho_1+\rho_2}
\quad \text{for}\quad |z|\to\infty,
\end{equation}
uniformly in $z\in\bar\Sigma_\vartheta$ for $\vartheta\in [0,\pi)$ fixed.
In particular there is a constant $N=N(\vartheta)$ such that
\begin{equation}
\label{k-estimate}
|k(z)|+|zk(z)|\le N,\quad z\in\Sigma_{\vartheta}.
\end{equation}
In the following we fix $\beta>0$.
For further analysis it will be convenient to introduce the
related extended symbol
\begin{equation}
\label{extended-symbol}
\tilde s(\lambda,\tau,\zeta):=
\lambda + \sigma \tau k(z)
+i\tau \zeta-[\![\rho]\!]\gamma_a k(z)/\tau,
\end{equation}
where $(\lambda,\tau)\in \Sigma_{\pi/2+\eta}\times\Sigma_\eta$
with $\eta$ sufficiently small, $z:=\lambda/\tau^2$,
and $\zeta\in U_{\beta,\delta}$
with 
$U_{\beta,\delta}:=\{\zeta\in\C: 
|\text{Re}\,\zeta|< \beta+1,\; |\text{Im}\,\zeta| < \delta\}$
and $\delta\in (0,1]$.
Clearly
$
\tilde s(\lambda,|\xi|,(b_0|\xi/|\xi|))=s_{b_0}(\lambda,\xi)
$
for $(\lambda,\xi)\in \Sigma_\eta\times \dot\R^n$.
\smallskip\\
We are going to show that for every fixed $\beta>0$
there are positive constants $\lambda_0$,
$\delta$, $\eta=\eta(\beta)$,
and $c_j=c_j(\beta,\lambda_0,\delta,\eta)$ such that
\begin{equation}
\label{tilde-s-estimate}
c_0\big[|\lambda |+|\tau|\big]\le |\tilde s(\lambda,\tau,\zeta)|
\le c_1\big[|\lambda |+ |\tau|\big],
\end{equation}
for all 
$(\lambda,\tau,\zeta)\in \Sigma_{\pi/2+\eta}\times\Sigma_{\eta}\times U_{\beta,\delta}$
with $|\lambda|\ge \lambda_0$.
The upper estimate is easy to obtain:
fixing $\vartheta\in (\pi/2,\pi)$ and $\lambda_0>0$, it follows
from~\eqref{k-estimate}
and the identity $k(z)/\tau=zk(z)\tau/\lambda$
that
\begin{equation}
\label{upper}
|\tilde s(\lambda,\tau,\zeta)|
\le |\lambda|+\big(\sigma N
+(\beta +2) +|[\![\rho]\!]|\gamma_a N/\lambda_0\big)|\tau|
\le c_1 \big[|\lambda |+|\tau|\big]
\end{equation}
for all 
$(\lambda,\tau,\zeta)\in \Sigma_{\pi/2+\eta}\times\Sigma_{\eta}
\times U_{\beta,\delta}$,
where $|\lambda |\ge \lambda_0$ and $\eta\in (0,\eta_0)$ with
$\eta_0:=(\vartheta-\pi/2)/3$.

In order to obtain a lower estimate we proceed as follows.
Suppose first that $\beta,\lambda_0>0$ are fixed and $\eta_0$ is as above.
Then we obtain
\begin{equation}
\label{lower-m}
\begin{split}
|\tilde s(\lambda,\tau,\zeta)|
&\ge |\lambda|
-\big(\sigma N+(\beta+2) +|[\![\rho]\!]|\gamma_a N/\lambda_0\big)|\tau|\\
&\ge (1/2)|\lambda|+(m/4)|\tau|
= c_0(\beta,\lambda_0) [|\lambda |+|\tau|],
\end{split}
\end{equation}
provided 
$(\lambda,\tau,\zeta)\in \Sigma_{\pi/2+\eta}\times\Sigma_{\eta}
\times U_{\beta,\delta}$,
$\eta\in (0,\eta_0)$, and
$|\lambda|\ge \lambda_0$
as well as $|\lambda|\ge m|\tau|$ with
\begin{equation*}
(m/4)\ge \sigma N +(\beta +2)+|[\![\rho]\!]|\gamma_a N/\lambda_0.
\end{equation*}
Next we will derive an estimate from below in case that
$|\lambda|\le M|\tau|^2$ with $M$ a positive constant.
From \eqref{k-properties} follows
that there are constants
$H,L, R>0$, depending on $M$,
such that
\begin{equation}
\label{k-M}
L\le \text{Re}\,(\sigma k(z))\le R,
\quad |\text{Im}\,(\sigma k(z))|\le H,
\end{equation}
whenever $(\lambda,\tau)\in \Sigma_{\pi/2+\eta}\times \Sigma_\eta$,
for $\eta\in (0,\eta_0)$
and $|\lambda|\le M|\tau|^2$,
where $z=\lambda/\tau^2$.
By choosing $\delta$ small enough we obtain from \eqref{k-M}
and the definition of $U_{\beta,\delta}$
\begin{equation*}
\begin{split}
 0<L-\delta\le \text{Re}\,(\sigma k(z)+i\zeta)\le R+\delta,\quad
|\text{Im}\,(\sigma k(z)+i\zeta)|\le H+(\beta+ 1)
\end{split}
\end{equation*}
provided
$(\lambda,\tau,\zeta)\in \Sigma_{\pi/2+\eta}\times\Sigma_\eta
\times U_{\beta,\delta}$,
$\eta\in (0,\eta_0)$ and 
$|\lambda|\le M|\tau|^2$, where $z=\lambda/\tau^2$.
By choosing $\eta$ small enough we conclude that there is
$\alpha=\alpha(M,\beta,\delta,\eta)\in (0,\pi/2)$ such that
\begin{equation}
\label{angle}
\tau (\sigma k(z)+i\zeta)\in \Sigma_\alpha
\end{equation}
whenever
$(\lambda,\tau,\zeta)\in\Sigma_{\pi/2+\eta}\times\Sigma_\eta\times U_{\beta,\delta}$
and $|z|\le M$ with $z=\lambda/\tau^2$.
We can additionally assume that $\eta$ is chosen so that 
$\psi:=\pi/2-\alpha-\eta>0$.
This implies
\begin{equation}
\begin{split}
|\tilde s(\lambda,\tau,\zeta)|
&\ge c(\psi)
\big[|\lambda| + |\tau|\,|\sigma k(z)+i\zeta|\big]
-|\tau||[\![\rho]\!]|\gamma_a N/\lambda_1 \\
&\ge c(\psi)\min(1,L-\delta)\big[|\lambda|+|\tau|\big]
-|\tau||[\![\rho]\!]|\gamma_a N/\lambda_1 \\
&\ge c_0(M,\beta,\lambda_1)\big[|\lambda|+|\tau|\big],
\end{split}
\end{equation}
provided 
$(\lambda,\tau,\zeta)\in\Sigma_{\pi/2+\eta}\times \Sigma_\eta\times U_{\beta,\delta}$,
$|\lambda|\le M|\tau|^2$ and $|\lambda|\ge \lambda_1$,
where $\lambda_1$ is chosen big enough.

Noting that the curves $|\lambda|=m|\tau|$ and $|\lambda|=M|\tau|^2$
intersect at $(m/M,m^2/M)$ we obtain
\eqref{tilde-s-estimate} by choosing 
$\lambda_0:=\max(\lambda_1, m^2/M)$.
\smallskip\\
(iii) 
In the following, we fix $\beta>0$ and we assume that
$b_0\in\R^n$ with $|b_0|\le\beta$.
Let then $S_{b_0}$ be the operator corresponding to the symbol $s_{b_0}$.
It is clear that $S_{b_0}$ is bounded from 
${_0\EE}_4(a)$ to ${_0\FF}_4(a)=:X$ and it remains to prove that it is boundedly invertible. For this we use the $\cH^\infty$-calculus and similar arguments
as in \cite[Section~4]{PrSi08} and \cite[Section~5]{PrSi09a}. 
First we note that $D_n$ admits an $\cR$-bounded $\cH^\infty$-calculus in $X$ with angle $0$; this follows from \cite[Theorem 4.11]{DHP03}. 
Therefore by the estimates obtained in~\eqref{tilde-s-estimate}, 
the operator family 
\begin{equation*}
\{(\lambda+D_n^{1/2})\tilde{s}^{-1}(\lambda, D_n^{1/2},\zeta):
(\lambda,\zeta)\in \Sigma_{\pi/2+\eta}\times U_{\beta,\delta},\;
|\lambda|\ge \lambda_0\}
\end{equation*}
is $\cR$-bounded.
Since $G=\partial_t$ is in $\cH^\infty(X)$ with angle $\pi/2$, 
the theorem of Kalton and Weis 
\cite[Theorem~4.4]{KW01} 
implies that the operator family 
$$
\{(G+D_n^{1/2})\tilde{s}^{-1}(G, D_n^{1/2},\zeta): \zeta\in U_{\beta,\delta}\}
$$
is bounded and holomorphic on $U_{\beta,\delta}$. Finally, we employ the Dunford calculus for the bounded linear operator $R_{b_0}:=(b_0|R)$, where $R$ denotes the Riesz operator with symbol $\xi/|\xi|$, $\xi\in\dot\R^n$.  
The operator $R_{b_0}$ is bounded and its spectrum is 
$\sigma(R_{b_0})=[-|b_0|,|b_0|]$, as e.g.\ the Mikhlin theorem shows. Since the operator family 
$$
\{(G+D_n^{1/2})\tilde{s}^{-1}(G, D_n^{1/2},\zeta):\zeta\in U_{\beta,\delta}\}
$$
is bounded and holomorphic in a neighborhood of $\sigma(R_{b_0})$, the classical Dunford calculus shows that the operator
$$
(G+D_n^{1/2})\tilde{s}^{-1}(G, D_n^{1/2},R_{b_0})
$$ 
is bounded in $X$,
uniformly for all $b_0\in\R^n$ with $|b_0|\le\beta$.
This shows that $S_{b_0}:{_0\EE}_4(a)\to{_0\FF}_4(a)$ is boundedly invertible,
uniformly for all $b_0\in\R^n$ with $|b_0|\le \beta$.

We emphasize that the bound for the 
operator $S_{b_0}^{-1}:{_0\FF}_4(a)\to{_0\EE}_4(a)$
depends only on the parameters $\rho_j$,
$\mu_j$, $\sigma$, $\gamma_a$, $p$ and $\beta$,
for $|b_0|\le \beta.$
\smallskip\\
\noindent
(iv) By means of a perturbation argument the result for constant $b$ can be extended to
variable $b=b_0+b_1(t,x)$. In fact,
given $\beta>0$ there exists a number  $\eta>0$ such that the
solution operator $S_b^{-1}$
exists and is bounded uniformly,  provided $|b_0|\le \beta$ and
$\|b_1\|_\infty+\|b_1\|_{{\FF}_4(a)}\le 2\eta.$ This follows
easily from the estimate
$$\|(b_1|\nabla h)\|_{{_0\FF}_4(a)}
\le c_0(\|b_1\|_\infty + \|b_1\|_{\FF_4(a)})\|h\|_{{_0\EE}_4(a)},$$
see Lemma \ref{le:3}(c).
\smallskip\\
\noindent
(v) In the general case we use a localization technique,
similar to \cite[Section~9]{AHS94}.
For this purpose we
first decompose $J$ into subintervals $J_k=[k\delta,(k+1)\delta]$
of length $\delta>0$ and solve the problem successively on these subintervals.
Since
$b\in BU\!C(J;C_0(\R^n,\R^n))$,
given any $\eta>0$ we may choose $\delta>0$ and $\varepsilon>0$
such that
\begin{equation*}
|b(t,x)-b(s,y)|\le \eta\quad\text{for all}\quad
(t,x), (s,y)\in J\times\R^n
\end{equation*}
with  $|t-s|\le \delta $ and $|x-y|_\infty\le\varepsilon$.
Let $\{U_j:=x_j+(\varepsilon/2)Q: j\in \N\}$ be an enumeration of the open covering
$\left\{({\varepsilon}/{2})({z}/{2}+Q):z\in \ZZ^n\right\}$
of $\R^n$, where $Q=(-1,1)^n$. Clearly,
\begin{equation}
\label{b-localized}
|b(t,x)-b(s,y)|\le \eta,\quad s,t\in J_k,\quad x,y\in U_j.
\end{equation}
Let $\phi$ be a smooth cut-off function with
support contained in  $(\varepsilon/2) Q$
such that $\phi\equiv1$ on $(\varepsilon/4)Q$.
Define
$$
\phi_j:=(\tau_{x_j}\phi)\Big(\sum_{k\in\N}(\tau_{x_k}\phi)^2\Big)^{-1/2},
\quad j\in\N,
$$
where $(\tau_{x_j}\phi)(x):=\phi(x-x_j)$.
Consequently, $\phi_j$ is a smooth cut-off function with
$\text{supp}(\phi_j)\subset U_j$ and $\sum_j \phi_j^2\equiv 1$.
For a function space
${\mathfrak F}(J;\R^n)\subset L_p(J;L_p(\R^n))$ we define
\begin{equation*}
\begin{aligned}
r (h_j):&=\sum_j\phi_jh_j,\quad &&(h_j)\in {\mathfrak F}(J;\R^n)^\nn,\\
  r^c h:&=(\phi_jh),\quad &&h\in{\mathfrak F}(J;\R^n).
\end{aligned}
\end{equation*}
Similarly as in \cite[Section 9]{AHS94} one shows that
\begin{equation}
\label{properties-r-rc}
r\in \Li(\ell_p({\mathfrak F}(J;\R^n)),{\mathfrak F}(J;\R^n)),
\
r^c\in \Li({\mathfrak F}(J;\R^n),\ell_p({\mathfrak F}(J;\R^n))),
\quad rr^c=I,
\end{equation}
for
${\mathfrak F}(J;\R^n)\in\{\FF_4(a), \EE_4(a)\}.$

Let $\theta$ be a smooth cut-off function with
$\text{supp}(\theta)\subset (\eps/2)Q$ such that $\theta\equiv 1$ on $\text{supp}(\phi)$
and let $\theta_j:=\tau_{x_j}\theta.$
Define
\begin{equation*}
b_{j,k}(t,x):=
\theta_j(x)\left(b(t,x)-b(k\delta,x_j)\right),
\quad (t,x)\in J\times\R^n.
\end{equation*}
It follows that
\begin{equation}
\label{bjk-small}
\|b_{j,k}\|_{BC(J_k\times\R^n)}+\|b_{j,k}\|_{\FF_4(J_k)}\le c_0\eta,
\quad k=0,\ldots,m,\quad j\in \nn,
\end{equation}
provided $\delta$ is chosen small enough.
Indeed, the estimates for $\|b_{j,k}\|_{BC(J_k,\R^n)}$
follow immediately from \eqref{b-localized}, while
the estimates for $\|b_{j,k}\|_{\FF_4(J_k)}$ can be shown
by approximating $b$ by functions that have better time regularity
and by carefully estimating the products
$\|\theta_j(b-b(k\delta,x_j))\|_{\FF_4(J_k)}$.

We now concentrate on the first interval $J_0=[0,\delta]$.
Let $L\in \Li({_0\EE}_4(a),{_0\FF}_4(a))$ denote
the operator with symbol $\sigma\tau k(z)$, i.e.\
$$L:=(\sigma D_n^{1/2}-[\![\rho]\!]\gamma_a D^{-1/2}_n)k(GD_n^{-1}):=L_1+L_2. $$
If follows from \eqref{bjk-small} and step (iii)
that the operator
\begin{equation*}
S_j:=G+L+(b(0,x_j)+b_{j,0}|\nabla):
{_0}\EE_4(\delta)\to {_0}\FF_4(\delta)
\end{equation*}
is invertible.
Moreover,
there is a constant $C_0$, depending only on
$\sup_j|b(0,x_j)|$ - and therefore only on $\|b\|_{BC(J\times \R^n)}$ -
such that
$
\|S_j^{-1}\|_{\Li(_{0}\FF_4(\delta),{_0}\EE_4(\delta))}\le C_0,
\ j\in\N.
$
\medskip\\
(vi)
Suppose that for a given $g\in {_0\FF}_4(\delta)$ we have a
solution $h\in {_0\EE}_4(\delta)$ of
\begin{equation*}
\label{redlinFBpert}
Gh +Lh +(b|\nabla)h=g.
\end{equation*}
Multiplying this equation by  $\phi_j$,
using that
$b\,\partial^\alpha\phi_j= (b(0,x_j)+ b_{j,0})\,\partial^\alpha \phi_j$
and  $rr^c=I$ this yields
\begin{equation*}
\label{lp-equation}
S_j\phi_jh-[L,\phi_j]h-(b|\nabla\phi_j) h
=\left(S_j-[L,\phi_j]r-(b|\nabla\phi_j)r\right)r^ch=r^cg,
\end{equation*}
where $[\cdot,\cdot ]$ denotes the commutator.
We now interpret this equation as an equation in
$\ell_p({_0}\FF_4(\delta))$.
It follows from step (iv) that
$(S_j)\in \text{Isom}(\ell_p({_0}\EE_4(\delta)),\ell_p({_0}\FF_4(\delta)))$ and
\begin{equation}
\label{S_j}
\|(S^{-1}_j)\|_{\Li(\ell_p({_0}\FF_4(\delta)),\ell_p({_0}\EE_4(\delta)))}\le C_0.
\end{equation}
We shall show below in step (vi) that the commutators satisfy
\begin{equation}
\label{commutators}
([L,\phi_j]+(b|\nabla\phi_j))\in \Li({_0}\FF(a),\ell_p(_{0}\FF_4(a))).
\end{equation}
Assuming this property, it follows from \eqref{properties-r-rc} that
\begin{equation*}
\|(([L,\phi_j]+(b|\nabla\phi_j))r(h_j))\|_{\ell_p({_0}\FF_4(\delta))}
\le C\|(h_j)\|_{\ell_p({_0}\FF_4(\delta))}
\le C\delta^\alpha\|(h_j)\|_{\ell_p({_0}\EE_4(\delta))}
\end{equation*}
for some $\alpha$ depending only on $p$ and $n$.
Therefore, choosing $\delta$ small enough
we can conclude that
$
\left(S_j-([L,\phi_j]+(b|\nabla \phi_j))r\right)
\in \text{Isom}(\ell_p(_{_0}\EE_4(\delta)),\ell_p({_0}\FF_4(\delta)))
$
with
\begin{equation*}
\|\left(S_j-([L,\phi_j]+(b|\nabla \phi_j))r\right)^{-1}\|\le 2C_0.
\end{equation*}
Let
$T_b:=r\left(S_j-([L,\phi_j]+(b|\nabla \phi_j))r\right)^{-1}r^c$.
Then $T_b\in\Li({_0}\FF_4(\delta),{_0}\EE_4(\delta))$ is a left inverse of
$S_b:=G+L+(b|\nabla)$.
Hence
\begin{equation}
\label{a-priori}
\|h\|_{{_0}\EE_4(\delta)}=
\|T_bS_bh\|_{{_0}\EE_4(\delta)}\le 2C_0\|r\|\,\|r^c\|\,\|S_bh\|_{{_0}\EE_4(\delta)},
\quad h\in {_0}\EE_4(\delta).
\end{equation}
Replacing $b$ by $\rho b$, $\rho \in[0,1]$,
we have a continuous family $\{S_{\rho b}\}$ of operators $S_{\rho b}$
which all satisfy the a-priori estimate \eqref{a-priori}
uniformly in $\rho\in [0,1]$.
Since $S_0$ is an isomorphism, we can infer from
a homotopy argument that $S_b$ is an isomorphism as well.
Repeating successively these arguments for the intervals $J_k$,
including the reduction from step (i), proves the assertion of the corollary.
\smallskip\\
(vii) We still have to verify the
estimate in \eqref{commutators}.
Since the covering $\{U_j:j\in\N\}$
has finite multiplicity, one obtains
 \begin{equation}
\label{partial-phi-g}
\|((\partial^\alpha\phi_j) g)\|_{\ell_p({_0}\FF_4(a))}
\le C(\alpha)\|g\|_{{_0}\FF_4(a)},
\quad g\in {_0}\FF_4(a).
\end{equation}
This together with Proposition \ref{le:3}(b) shows that
\begin{equation*}
\|(b|\nabla \phi_j)h\|_{\ell_p(_{0}\FF_4(a)}
\le C\|bh\|_{{_0}\FF_4(a)}
\le C_0 (\|b\|_\infty+\|b\|_{\FF_4(a)})\|h\|_{{_0}\FF_4(a)}.
\end{equation*}
The estimates for the commutators $[L,\phi_j]$ are more involved.
The operator $A=GD_n^{-1}$ with canonical domain is sectorial and admits a bounded $\cH^\infty$-calculus
with angle $\pi/2$ in ${_0H}^s_p(J;K^r_p(\R^n))$, for $K\in\{H,W\}$,
and also in ${_0W}^s_p(J;K^r_p(\R^n))$ by real interpolation.
Hence fixing $\theta\in (0,\pi/2)$,  the following resolvent estimate holds
in these spaces:
$$\|z(z-A)^{-1}\|\le M,\quad \mbox{ for all } z\in -\Sigma_{\theta}.$$
The function $k(z)$ is holomorphic in $\C\setminus (-\infty, - 2\delta_0]$
for some $\delta_0>0$   and behaves like $1/z$ as $|z|\to\infty$. Choose the contour
$$\Gamma=(\infty,\delta_0]e^{i\psi}\cup\delta_0 e^{i[\psi,2\pi-\psi]}\cup[\delta_0,\infty)e^{-i\psi},$$
where $\pi>\psi >\pi-\theta$. Then we have the Dunford integral
$$k(A)=\frac{1}{2\pi i} \int_\Gamma k(z)(z-A)^{-1}dz,$$
which is absolutely convergent.
This shows that $k(A)$ is bounded, as is $Ak(A)$ 
thanks to $A\in \cH^\infty$, thus $A^{1/2}k(A)$ is bounded as well.
Therefore the identity $k(A)D_n^{-1/2}= G^{-1/2} A^{1/2}k(A)$ shows that $L_2$ is bounded since $G^{-1/2}$ is, 
and~\eqref{commutators} follows for $[L_2,\phi_j]$.
For the commutator $[L_1,\phi_j]$ we obtain
$$[L_1,\phi_j] = \sigma[k(A)D_n^{1/2},\phi_j]=\sigma[k(A),\phi_j]D_n^{1/2} +\sigma k(A)[D^{1/2}_n,\phi_j].$$
Using the Dunford integral for $k(A)$ this yields
$$[k(A),\phi_j] = \frac{1}{2\pi i} \int_\Gamma k(z)[(z-A)^{-1},\phi_j]dz
              =\frac{1}{2\pi i} \int_\Gamma k(z)(z-A)^{-1}[A,\phi_j](z-A)^{-1}dz,$$
hence with
\begin{align*}
[A,\phi_j]&= GD_n^{-1}[\phi_j,D_n]D_n^{-1}= A(\Delta\phi_j +
2(\nabla\phi_j|\nabla))D_n^{-1}\\
          &= A(\Delta \phi_j D^{-1}_n +2i(\nabla\phi_j|R)D^{-1/2}_n),
\end{align*}
we have
$$
[k(A),\phi_j]D_n^{1/2} = \frac{1}{2\pi i} \int_\Gamma k(z)A(z-A)^{-1}
\{-\Delta\phi_j G^{-1/2} A^{1/2} +2i(\nabla\phi_j|R) \} (z-A)^{-1}dz.
$$
Let $h\in{_0}\FF_4$ be given.
Then we obtain from
\begin{equation*}
\|k(z)A(z-A)^{-1}\|_{\Li({_0}\FF_4)}\le C/|z|,
\quad \|A^{1/2}(z-A)^{-1}\|_{\Li({_0}\FF_4)}\le C/|z|^{1/2}, \quad z\in\Gamma,
\end{equation*}
from \eqref{partial-phi-g},
and from Minkowski's inequality for integrals
\begin{equation*}
\begin{split}
&\Big\|\Big(\int_\Gamma k(z)A(z-A)^{-1}
\Delta\phi_j G^{-1/2} A^{1/2} (z-A)^{-1}h\,dz\Big)\Big\|_{\ell_p({_0}\FF_4)}\\
&\le C \int_\Gamma \frac{1}{|z|}
\|(\Delta\phi_jG^{-1/2}A^{1/2}(z-A)^{-1}h)\|_{\ell_p({_0}\FF_4)}\,|dz|\\
&\le
C\int_\Gamma\frac{1}{|z|^{3/2}}\|h\|_{{_0}\FF_4}\,|dz|
\le C\|h\|_{{_0}\FF_4}
\end{split}
\end{equation*}
where we also used that $G^{-1/2}$ is bounded on compact intervals.
In the same way we can estimate
the second term in the integral representation
of $[k(A),\phi_j]D^{1/2}$,
this time using the fact that $R$ is  bounded.

To estimate the commutators $[D_n^{1/2},\phi_j]$ in ${_0\FF}_4$
 note that
 \begin{equation*}
 \begin{split}
 (D_n)^{1/2} = D_n(D_n)^{-1/2} &= \frac{1}{\sqrt \pi} D_n \int_0^\infty e^{-D_n t}  t^{-\frac{1}{2}}\,dt \\
&=\frac{1}{\sqrt \pi}
\left( D_n \int_0^1 e^{-D_n t}  t^{-\frac{1}{2}}\,dt
+ D_n \int_1^\infty e^{-D_n t}t^{-\frac{1}{2}}\,dt\right) \\
&=: \frac{1}{\sqrt \pi} (T_1 + T_2),
\end{split}
\end{equation*}
with $e^{-D_n t}$ denoting the bounded analytic semigroup generated by the Laplacian in $H^s_p(\R^n)$ which extends by real interpolation  to $W^s_p(\R^n)$,
 and then canonically to ${_0\FF}_4$. Thus by \eqref{partial-phi-g}
 there is a constant $C>0$ such that for $h \in {_0\FF}_4$ we have
 \begin{equation*}
 \begin{split}
 \|(\phi_j T_2 h)\|_{\ell_p({_0\FF}_4)} &= \|(\phi_j \int_1^\infty D_n e^{-D_n t}  t^{-\frac{1}{2}} h\,dt)\|_{\ell_p({_0\FF}_4)} \\
 &\le C \|\int_1^\infty D_n e^{-D_n t}  t^{-\frac{1}{2}} h\,dt\|_{{_0\FF}_4} \\
 &\le C\int_1^\infty   t^{-\frac{3}{2}}\,dt\,\|h\|_{{_0\FF}_4}
 \le C \|h\|_{{_0\FF}_4}, \\
\|(T_2 \phi_j h)\|_{\ell_p({_0\FF}_4)} &= \|\int_1^\infty D_n e^{-D_n t}  t^{-\frac{1}{2}} (\phi_jh) \,dt\|_{\ell_p({_0\FF}_4)} \\
&\le C  \|(\phi_jh)\|_ {\ell_p({_0\FF}_4)} \le C \|h\|_{{_0\FF}_4} .
\end{split}
\end{equation*}
 Hence $\|([\phi_j,T_2]h)\|_{\ell_p({_0\FF}_4)} \le C \|h\|_{{_0\FF}_4}$.

We consider next the commutator $[T_1,\phi_j]$.
Let $k_t(x)=(2\pi t)^{-n/2}\exp(-|x|^2/4t)$ denote the Gaussian kernel. Then for fixed $t>0$, the operator $D_ne^{-D_nt}$ is the convolution with kernel $-\Delta k_t(x)$, which is of class $C^\infty$.
It is not difficult to see that there are constants $C,c>0$ such that
\begin{equation}
\label{heat-kernel}
|\Delta k_t(x)|\le C {t^{-(n+2)/2}}\;e^{-c|x|^2/t},\quad x\in\R^n,\ t>0.
\end{equation}
Choosing a cut-off function $\chi\in C^\infty(\R^n)$ with
$\chi\equiv1$ in $B_\rho(0)$, $\text{supp}\,(\chi)\subset B_{2\rho}(0)$ and
$0\le \chi\le 1$ elsewhere, we set
$$-\Delta k_t(x)= -(1-\chi(x))\Delta k_t(x)-\chi(x)\Delta k_t(x)
=: k_{3,t}(x)+k_{4,t}(x),\quad x\in \R^n, \; t>0,$$
and we denote by $T_l$ the convolution operators with
kernels $\int_0^1k_{l,t}t^{-1/2}dt$, $l=3,4$. For the kernel of $T_3$ we obtain
from \eqref{heat-kernel}
the estimate
\begin{equation*}
\begin{split}
|\int_0^1 k_{3,t}(x)t^{-1/2}\,dt|
&\le C\int_0^1 e^{-c|x|^2/t} t^{-(n+3)/2}\,dt\\
&\le C e^{-c_1|x|^2}\int^\infty_1 e^{-c_2|x|^2s}s^{(n-1)/2}ds
\le C e^{-c_1|x|^2},
\end{split}
\end{equation*}
as $k_{3,t}(x)=0$ for $|x|\le \rho$.
Thus this kernel is in $L_1(\R^n)$ and hence we may estimate the commutator $[T_3,\phi_j]$ in the same way as $[T_2,\phi_j]$.

For the remaining commutator $[T_4,\phi_j]$
note that
$$\partial^\alpha[T_4,\phi_j]= \sum_{\beta\le\alpha} \binom{\alpha}{\beta}
 [T_4, \partial^\beta\phi_j]\partial^{\alpha-\beta}.$$
This shows that it is enough to estimate the commutator $[T_4,\phi_j]$ in $L_p(\R^n)$, as it then extends to $H^m_p(\R^n)$ and by interpolation to $W^s_p(\R^n)$, and then canonically to ${_0\FF}_4$.
Next we observe that for $x,y \in \R^n$
\begin{equation*}
\partial^\alpha\phi_j(y)-\partial^\alpha\phi_j(x)
= \partial^\alpha\phi^\prime_j(x)(y-x)+r_{j,\alpha}(x,y),
\end{equation*}
where $|r_{j,\alpha}(x,y)|\le C|x-y|^{2},$
with some constant $C$ independent of $j$ and $|\alpha|\le2$.
Therefore
\begin{equation*}
\begin{split}
[T_4,\phi_j]h(x) &=
\int_0^1 \int_{\R^n} (\phi_j(y)-\phi_j(x))k_{4,t}(x-y)h(y)\,dy \,t^{-\frac{1}{2}}\,dt \\
&= -\phi^\prime_j(x) \int_{\R^n}\int_0^1
(y-x) k_{4,t}(x-y)t^{-\frac{1}{2}} \,dt\,h(y)\,dy + \\
&\qquad\qquad+  \int_{\R^n}\int_0^1 r_j(x,y) k_{4,t}(x-y)t^{-\frac{1}{2}} \,dt\,h(y)\,dy \\
 &=:T_{5,j}h(x) + T_{6,j}h(x).
\end{split}
\end{equation*}
 We observe that the support of the kernel $k_{4,t}$ is contained in $B_{2\rho}(0)$, and consequently we may replace $h$ by $\psi_jh$, where $\psi_j$
 is a cut-off function which equals 1 on
 $\text{supp}(\phi_j)+B_{2\rho}(0)$.
 In the following we fix a smooth cut-off function $\psi$ which equals $1$ on
 $\text{supp}(\phi)+B_{2\rho}(0)$ and then set $\psi_j:=\tau_{x_j}\psi$.
 We then have
 $$
 \|(T_{l,j}h)\|_{\ell_p(L_p)} = \|(T_{l,j}\psi_jh)\|_{\ell_p(L_p)}\le \sup_k\|T_{l,k}\|_{\Li(L_p)}\|(\psi_jh)\|_{\ell_p(L_p)}\le C \|h\|_{L_p},
 $$
 provided we can show that the operators $T_{l,k}$ are $L_p$-bounded with bound independent of $k\in\N$ for $l=5,6$.

The operators $T_{l,j}$ satisfy
\begin{equation*}
\begin{aligned}
 T_{5,j} h &= \phi^\prime_j(q * h) \quad \hbox{with} \ &q(x) &= \chi(x)\int_0^1 x \Delta k_t(x)t^{-\frac{1}{2}} \,dt, \, x \in \R^n  \\
 |T_{6,j} h| &\le  r * |h| \qquad \hbox{with} \ &r(x) &= C\chi(x)\int_0^1
 |x|^{2} |\Delta k_t(x)|t^{-\frac{1}{2}} \,dt, \, x \in \R^n. \\
\end{aligned}
\end{equation*}
The Fourier transform of $q$ is given by
$\widehat{q}(\xi) = C\widehat{\chi}*
\int_0^1 \nabla_\xi
 (|\xi|^2 e^{-t|\xi|^2})t^{-1/2}dt $ and we verify that
$$
\sup_{\alpha\le (1,\ldots,1)}\sup_{\xi \in \R^n}|\xi|^{|\alpha|}
| \partial^\alpha \widehat{q}(\xi)| \le M
$$
for some $M <\infty$. It thus follows from Mikhlin's multiplier theorem
that
\begin{equation*}
\|T_{5,j}h\|_{L_p}
\le C\|\phi^\prime_j\|_\infty \|h\|_{L_p}
\le C\|h\|_{L_p}.
\end{equation*}
Finally, in order to estimate
$T_{6,j}$ we infer from \eqref{heat-kernel}
that
\begin{equation*}
 r(x) \le C \int_0^1 |x|^{2}  e^{-c|x|^2/t}\;{t^{-\frac{n+3}{2}}}\, dt
\le C e^{-c_1|x|^2}|x|^{-(n-1)}\int_1^\infty e^{-c_2s}s^{(n-1)/2}ds
\end{equation*}
for $x \in \R^n$.
It follows that  $r\in L_1(\R^n)$ which implies by Young's inequality
  $\|T_{6,j} h\|_p \le C \|h\|_p$ with a uniform constant $C$.
\end{proof}

\begin{rems}
(a) We mention that the proof for the estimate of
$[D_n^{1/2},\phi_j]$ follows the ideas of \cite[Lemma 6.4]{DDHPV}.
\smallskip\\
(b) If $\rho_2\le\rho_1$, i.e. the light fluid lies above the heavy one,
then the estimate~\eqref{tilde-s-estimate} can be improved in the following sense:
for every $\beta>0$ and $\lambda_0>0$ there
are positive constants $\delta$,
$\eta=\eta(\beta)$ and $c_j=c_j(\beta,\lambda_0,\delta,\eta)$ such that
\begin{equation}
\label{stable-estimate}
     c_0\big[|\lambda|+|\tau|\big]\le \tilde s(\lambda,\tau,\zeta)
 \le c_1\big[|\lambda|+|\tau|\big]
\end{equation}
for all 
$(\lambda,\tau,\zeta)\in\Sigma_{\pi/2+\eta}\times\Sigma_\eta\times U_{\beta,\delta}$
and $|\lambda|\ge \lambda_0$.
For this we observe that estimates
\eqref{upper} and \eqref{lower-m} certainly also hold 
in case that $\rho_2\le\rho_1$.
On the other hand, given $M>0$ 
we conclude as in \eqref{k-M} that
$L\le \text{Re}\,((\rho_1-\rho_2)\gamma_a k(z))\le R$
and $|\text{Im}\,((\rho_1-\rho_2)\gamma_a k(z))|\le H$,
with appropriate positive constants $L,R,H$.   
This shows that there 
exists $\alpha=\alpha(M,\eta)\in (0,\pi/2)$ such that
\begin{equation}
\label{angle2}  
(\rho_1-\rho_2)\gamma_a k(z)/\tau\in \Sigma_\alpha,
\quad (\lambda,\tau)\in\Sigma_{\pi/2+\eta}\times\Sigma_{\eta},
\quad |z|\le M
\end{equation}
with $\eta\in (0,\eta_0)$ chosen small enough,
where we can assume that $\alpha$ coincides with 
the angle in \eqref{angle}. 
Combining \eqref{angle} and \eqref{angle2} yields
\begin{equation*}
\begin{split}
|\tilde s(\lambda,\tau,\zeta)|
&\ge c(\psi)\big[|\lambda|+|\tau(\sigma k(z) +i\zeta)
+(\rho_1-\rho_2)\gamma_a k(z)/\tau|\big]\\
&\ge c(\psi)c(\alpha)\big[|\lambda|+|\tau(\sigma k(z)+i\zeta)|
+|(\rho_1-\rho_2)\gamma_a k(z)/\tau|\big]\\
&\ge c_0(M,\beta,\delta,\eta)\big[|\lambda|+|\tau|]
\end{split}
\end{equation*}
provided $(\lambda,\tau,\zeta)\in\Sigma_{\pi/2+\eta}\times\Sigma_\eta
\times U_{\beta,\delta}$ and $|\lambda|\le M |\tau|^2$.
Noting again that the curves $|\lambda|=m|\tau|$ and $|\lambda|=M|\tau|^2$
intersect at $(m/M,m^2/M)$ we obtain
\eqref{stable-estimate} by choosing $M$ big enough.
\smallskip\\
(c) If $\rho_2\le\rho_1$ we can conclude from
the lower estimate in \eqref{stable-estimate} that
the function $\tilde s$ does not have zeros in 
$\overline\Sigma_{\pi/2}\times\R_+\times [-\beta,\beta]$.
This holds in particular true for
the symbol $s(\lambda,\tau):=\tilde s(\lambda,\tau,0)$,
indicating that there are no instabilities in case
that the light fluid lies on top of the heavy one.
\smallskip\\
(d)
If $\rho_2>\rho_1$ then it is shown in \cite{PrSi09c} that the symbol
$s$ has for each $\tau\in (0,\tau_\ast)$ with 
$\tau_*:=((\rho_2-\rho_1)\gamma_a/\sigma)^{1/2}$
a zero $\lambda=\lambda(\tau)>0$, pertinent to the Rayleigh-Taylor instability.
\smallskip\\
(e) Further mapping properties of the boundary symbol 
$s(\lambda,\tau):=\tilde s(\lambda,\tau,0)$
and the associated operator $S$ in case that $\gamma_a=0$ 
have been derived in
\cite{PrSi08}. In particular, we have investigated
the singularities and zeros of $s$, and
we have studied the mapping properties of $S$ in case
of low and high frequencies, respectively.
\smallskip\\
\end{rems}
\section{The nonlinear problem}
In this section we prove existence und uniqueness
of solutions for the nonlinear problem \eqref{tfbns2},
and we show additionally that solutions
immediately regularize and are real
analytic in space and time.
In order to facilitate this
task, we first introduce some notation.
We set
\begin{equation}
\label{EE}
\begin{split}
&\EE_1(a):= \{u\in H^1_p(J;L_p(\R^{n+1},\R^{n+1}))
\cap L_p(J;H^2_p(\dot\R^{n+1},\R^{n+1})):\: [\![u]\!]=0\}, \\
&\EE_2(a):=L_p(J;\dot H^1_p(\dot\R^{n+1})), \\
&\EE_3(a):=
W^{1/2-1/2p}_p(J;L_p(\R^n)) \cap L_p(J;W^{1-1/p}_p(\R^n)), \\
&\EE_4(a):= W^{2-1/2p}_p(J;L_p(\R^{n}))
\cap H^1_p(J;W^{2-1/p}_p(\R^n))\\
&\hspace{1.3cm}\cap W^{1/2-1/2p}_p(J;H^2_p(\R^n))
\cap L_p(J;W^{3-1/p}_p(\R^n)),\\
&\EE(a)_{\phantom{3}}:= \{(u,\pi,q,h) \in \EE_1(a)\times
\EE_2(a)\times \EE_3(a)\times\EE_4(a):\: [\![\pi]\!]=q\}.
\end{split}
\end{equation}
The space $\EE(a)$ is given the natural norm
\begin{equation*}
\|(u,\pi,q,h)\|_{\EE(a)}
=\|u\|_{\EE_1(a)}+\|\pi\|_{\EE_2(a)}+\|q\|_{\EE_3(a)}+\|h\|_{\EE_4(a)}
\end{equation*}
which turns it into a Banach space.
Moreover, we set
\begin{equation}
\label{FF}
\begin{split}
&\FF_1(a):=L_p(J;L_p(\R^{n+1},\R^{n+1})), \\
&\FF_2(a):=H^1_p(J;H^{-1}_p(\R^{n+1}))\cap L_p(J;H^1_p(\dot\R^{n+1})), \\
&\FF_3(a):=W^{1/2-1/2p}_p(J;L_p(\R^n,\R^{n+1}))
\cap  L_p(J;W^{1-1/p}_p(\R^n,\R^{n+1})), \\
&\FF_4(a):= W^{1-1/2p}_p(J;L_p(\R^{n}))\cap L_p(J;W^{2-1/p}_p(\R^{n})),\\
&\FF(a)_{\phantom{3}}:=\FF_1(a)\times \FF_2(a)\times
\FF_3(a)\times \FF_4(a).
\end{split}
\end{equation}
The generic elements of $\FF(a)$ are the functions $(f,f_d,g,g_h)$.
\medskip\\
Let $b\in \FF_4(a)^n$ be a given function.
Then we define the nonlinear mapping
\begin{equation}
\label{K} N_b(u,\pi,q,h):=\big(F(u,\pi,h),F_d(u,h), G(u,q,h),
(b-\gamma v|\nabla h)\big)
\end{equation}
for $(u,\pi,q,h)\in\EE(a)$, where, as before, $u=(v,w)$,
$F=(F_v,F_w)$ and
$G=(G_v,G_w)$. We will now study the mapping properties of $N_{b}$ and
we will derive estimates for the Fr\'echet derivative of $N_{b}$.
\goodbreak
\begin{prop}
\label{pro:estimates-K} Suppose $p>n+3$ and  $b\in \FF_4(a)^n$.  Then
\begin{equation}
\label{K-analytic}
N_{b}\in C^\omega(\EE(a)\,,\FF(a)),\quad a>0.
\end{equation}
Let $DN_{b}(u,\pi,q,h)$ denote the Fr\'echet derivative of
$N_{b}$ at $(u,\pi,q,h)\in \EE(a)$. Then
$ DN_{b}(u,\pi,q,h) \in\Li({_0}\EE(a), {_0}\FF(a))$,
and for any number $a_0>0$
there is a positive constant
$M_0=M_0(a_0,p)$ such that
\begin{equation*}
\label{Frechet}
\begin{split}
 &\|DN_{b}(u,\pi,q,h)\|_{\Li({_0}\EE(a),\, {_0}\FF(a))} \\
 &\le M_0\big[
  \|b-\gamma v\|_{BC(J;BC)\,\cap\, \FF_4(a)}
 +\|(u,\pi,q,h)\|_{\EE(a)}  \big]\\
& + M_0\big[\big(\|\nabla h\|_{BC(J;BC^1)}
+\|h\|_{\EE_4(a)}+\|u\|_{BC(J;BC)}\big)\|u\|_{\EE_1(a)}\big]\\
&+ M_0\big[ P(\|\nabla h\|_{BC(J;BC)})\|\nabla h\|_{BC(J;BC)}
+Q\big(\|\nabla h\|_{BC(J;BC^1)},\|h\|_{\EE_4(a)}\big)\|h\|_{\EE_4(a)}\big]
\end{split}
\end{equation*}
for all $(u,\pi,q,h)\in \EE(a)$ and all $a\in (0,a_0]$.
Here, $P$ and $Q$ are fixed polynomials with coefficients
equal to one.
\end{prop}
\begin{proof}
The proof of the proposition is relegated to the end of the appendix.
\end{proof}

Given $h_0\in W^{3-2/p}_p(\R^n)$ we define
\begin{equation}
\label{Theta}
\Theta_{h_0}(x,y):=(x,y+h_0(x)),\quad (x,y)\in \R^n\times \R.
\end{equation}
Letting $\Omega_{h_0,i}:=\{(x,y)\in\R^n\times \R: (-1)^i(y-(h_0(x))>0\}$
and $\Omega_{h_0}:=\Omega_{h_0,1}\cup \Omega_{h_0,2}$
we obtain from Sobolev's embedding theorem that
\begin{equation*}
\Theta_{h_0}\in
\text{Diff}^2 (\dot\R^{n+1},\Omega_{h_0})
\cap\text{Diff}^2 (\R^{n+1}_{-},\Omega_{h_0,1})
\cap \text{Diff}^2(\R^{n+1}_{+},\Omega_{h_0,2}),
\end{equation*}
i.e., $\Theta_{h_0}$ yields a $C^2$-diffeomrphism between
the indicated domains.
The inverse transformation obviously is given by
$\Theta^{-1}_{h_0}(x,y)=(x,y-h_0(x))$.
It then follows from the chain rule and the transformation rule for
integrals that
\begin{equation*}
\Theta^\ast_{h_0}\in\text{Isom}
(H^k_p(\dot\R^{n+1}),H^k_p(\Omega_{h_0})),
\quad [\Theta^\ast_{h_0}]^{-1}=\Theta^{h_0}_\ast,
\quad k=0,1,2,
\end{equation*}
where we use the notation
\begin{equation*}
\begin{split}
&\Theta^\ast_{h_0}u:=u\circ\Theta_{h_0},\quad u:\Omega_{h_0}\to \R^m, \\
&\Theta^{h_0}_\ast v:=v\circ\Theta^{-1}_{h_0},\quad v:\dot\R^{n+1}\to\R^m.
\end{split}
\end{equation*}
We are now ready to prove our main result of this section.

\goodbreak
\begin{thm}
\label{th:nonlinear}
{\rm (}Existence of solutions for the
nonlinear problem \eqref{tfbns2}{\rm)}.
\begin{itemize}
\item[(a)] For every $\beta>0$ there exists a constant
$\eta=\eta(\beta)>0$ such that for all initial values
$$(u_0,h_0)\in W^{2-2/p}_p(\dot\R^{n+1},\R^{n+1})
\times W^{3-2/p}_p(\R^n)\quad\text{with}\quad [\![u_0]\!]=0,
$$
satisfying the compatibility conditions
\begin{equation}
\label{compatibility}
[\![\mu D(\Theta^{h_0}_\ast u_0)\nu_0
-\mu(\nu_0| D(\Theta^{h_0}_\ast u_0)\nu_0)\nu_0]\!]=0,
\quad div(\Theta^{h_0}_\ast u_0)=0,
\end{equation}
and the smallness-boundedness condition
\begin{equation}
\label{sm}
\|\nabla h_0\|_\infty\le \eta,\quad \|u_0\|_\infty\le \beta,
\end{equation}
there is a number $t_0=t_0(u_0,h_0)$ such that
the nonlinear problem \eqref{tfbns2}
admits a unique solution $(u,\pi,[\![\pi]\!],h)\in \EE_1(t_0)$.

\vspace{1mm}
\item[(b)] The solution has the additional regularity properties
\begin{equation}
(u,\pi)\in C^\omega((0,t_0)\times\dot\R^{n+1},\R^{n+2}),\quad
[\![ \pi ]\!],h\in C^\omega((0,t_0)\times\R^n). \\
\end{equation}
In particular, ${\mathcal M}=\bigcup_{t\in (0,t_0)}\big(\{t\}\times
\Gamma(t)\big) \text{ is a real analytic manifold}$.
\end{itemize}
\end{thm}
\begin{proof}
The proof of this result proceeds in a similar way
as the proof of Theorem~6.3 in \cite{PrSi09a}.

For a given function $b\in\FF_4(a)^n$
we consider the nonlinear problem
\begin{equation}
\label{nonlinear-b}
\left\{
\begin{aligned}
\rho\partial_tu -\mu\Delta u+\nabla \pi&= F(u,\pi,h)
    &\ \hbox{in}\quad &\dot\R^{n+1}\\
{\rm div}\,u&= F_d(u,h)&\ \hbox{in}\quad &\dot\R^{n+1}\\
-[\![\mu\partial_y v]\!] -[\![\mu\nabla_{x}w]\!] &=G_v(u,[\![\pi]\!],h)
    &\ \hbox{on}\quad &\R^n\\
-2[\![\mu\partial_y w]\!] +[\![\pi]\!] -\sigma\Delta h &= G_w(u,h)
    &\ \hbox{on}\quad &\R^n\\
[\![u]\!] &=0 &\ \hbox{on}\quad &\R^n\\
\partial_th-\gamma w+(b|\nabla h)&=(b-\gamma v|\nabla h) &\ \hbox{on}\quad &\R^n\\
u(0)=u_0,\; h(0)&=h_0, \\\end{aligned}
\right.
\end{equation}
which clearly is equivalent to \eqref{tfbns2}.
\smallskip\\
In order to  economize our notation we set $z:=(u,\pi,q,h)$ for
$(u,\pi,q,h)\in\EE(a)$. With this notation, the nonlinear
problem~\eqref{tfbns2} can be restated as
\begin{equation}
\label{FP-1} L_{b}z=N_{b}(z),\quad (u(0),h(0))=(u_0,h_0),
\end{equation}
where $L_{b}$ denotes the linear operator on the left-hand side of
\eqref{nonlinear-b}, and $N_{b}$ correspondently denotes
the nonlinear mapping on the right-hand site of \eqref{nonlinear-b}.

It is convenient to first introduce an auxiliary function
$z^\ast=z_b^\ast\in\EE(a)$ which resolves
the compatibility conditions
and the initial conditions in
\eqref{FP-1}, and then to solve the resulting reduced
problem
\begin{equation}
\label{FP-2} L_{b}z=N_{b}(z+z^\ast)-L_{b}z^\ast=:K_{b}(z), \quad z\in
{_0}\EE(a),
\end{equation}
by means of a fixed point argument.
\smallskip\\
(i) Suppose that $(u_0,h_0)$ satisfies the
(first) compatibility condition in \eqref{compatibility}, and let
\begin{equation*}
[\![\pi_0]\!]:=
\theta^\ast_{h_0}\{[\![\mu(\nu_0 |D(\Theta^{h_0}_\ast u_0)\nu_0)]\!]+\sigma\kappa\},
\end{equation*}
where $\theta_{h_0}:=\Theta_{h_0}|_{\R^{n+1}\times\{0\}}$.
Here we observe that
$\theta^\ast_{h_0}[\![\omega]\!]=[\![\Theta^\ast_{h_0}\omega]\!]$
for any function $\omega:\Omega_{h_0}\to\R^m$
which has one-sided limits.
It is then clear from  the definition in \eqref{2.3}--\eqref{2.4}
that the following compatibility conditions hold:
\begin{equation}
\label{comp-2}
\begin{aligned}
-[\![\mu\partial_y v_0]\!] -[\![\mu\nabla_{x}w_0]\!]
 &=G_v(u_0,[\![\pi_0]\!],h_0)
    &\ \hbox{on}\quad &\R^n\\
-2[\![\mu\partial_y w_0]\!] +[\![\pi_0]\!] -\sigma\Delta h_0 &=
G_w(u_0,h_0)
    &\ \hbox{on}\quad &\R^n\\
\end{aligned}
\end{equation}
where, as before, $u_0=(v_0,w_0)$. Next we introduce special
functions $(0,f^\ast_d,g^\ast,g^\ast_h)\in\FF(a)$ which resolve
the necessary compatibility conditions. First we set
\begin{equation}
\label{c-star} c^\ast(t):= \left\{
\begin{aligned}
&{\mathcal R}_{+} e^{-t D_{n+1}}\mathcal{E}_{+}(v_0|\nabla
h_0)
\quad\text{in}\quad \R^{n+1}_{+}, \\
&{\mathcal R}_{-} e^{-t D_{n+1}}\mathcal{E}_{-}(v_0|\nabla h_0)
\quad\text{in}\quad \R^{n+1}_{-},
\end{aligned}
\right.
\end{equation}
where ${\mathcal E}_{\pm}\in
\Li(W^{2-2/p}_p(\R^{n+1}_{\pm}),W^{2-2/p}_p(\R^{n+1}))$ is an
appropriate extension operator and ${\mathcal R}_{\pm}$ is the
restriction operator. Due to $(v_0|\nabla h_0)\in
W^{2-2/p}_p(\dot\R^{n+1})$ we obtain
\begin{equation*}
c^\ast\in H^1_p(J;L_p(\R^{n+1}))\cap
L_p(J;H^2_p(\dot\R^{n+1})).
\end{equation*}
Consequently,
\begin{equation}
\label{divergence-star} f^\ast_d:=\partial_y\, c^\ast\in \FF_2(a)
\ \text{ and }\  f^\ast_d(0)=F_d(v_0,h_0).
\end{equation}
Next we set
\begin{equation}
\label{star}
\begin{split}
g^\ast(t):=e^{-D_nt} G(u_0,[\![\pi_0]\!],h_0). \quad
g^\ast_h(t):=e^{-D_nt}(b(0)-\gamma v_0|\nabla h_0).
\end{split}
\end{equation}
It then follows from
\eqref{divergence-star} and \cite[Lemma 8.2]{EPS03} that
$(0,f^\ast_d,g^\ast,g^\ast_h)\in\FF(a)$. \eqref{comp-2} and the
second condition in \eqref{compatibility} show
that the necessary compatibility conditions of
Theorem~\ref{th:b}
are satisfied and we can conclude that the
linear problem
\begin{equation}
\label{FP-3} L_{b}z^\ast=(0,f^\ast_d,g^\ast,g^\ast_h), \quad
(u^\ast(0),h^\ast(0))=(u_0,h_0),
\end{equation}
has a unique solution $z^\ast=z_b^\ast\in\EE(a)$.
 With the auxiliary function $z^\ast$ now determined, we can
focus on the reduced equation \eqref{FP-2}, which can be converted
into the fixed point equation
\begin{equation}
\label{FP-4} z=L_{b}^{-1}K_{b}(z),\quad z\in{_0}\EE(a).
\end{equation}
 Due to the choice of
$(f^\ast_d,g^\ast,g^\ast_h)$ we have $K_{b}(z)\in {_0}\FF(a)$ for
any $z\in{_0}\EE(a)$, and it follows from
Proposition~\ref{pro:estimates-K} that
\begin{equation*}
\label{FP-6} K_{b}\in C^\omega({_0}\EE(a),{_0}\FF(a)).
\end{equation*}
Consequently, $L_{b}^{-1}K_{b}:{_0}\EE(a)\to {_0}\EE(a)$ is well defined
and smooth.
\medskip\\
(ii)
An inspection of the proof of Theorem~\ref{th:b}
shows that given $\beta>0$
we can find a positive number $\delta_0=\delta_0(b) $
such that
\begin{equation}
\label{FP-5-uniform}
L_{b}^{-1}\in \Li({_0}\FF(a),{_0}\EE(a)),\quad
\|L_{b}^{-1}\|_{\Li({_0}\FF(a),{_0}\EE(a))}\le M,
\quad a\in [0,\delta_0],
\end{equation}
whenever $b\in \FF_4(a)^n$ and $\|b\|_{BC[0,a],BC(\R^n))}\le \beta$.
It should be pointed out that the bound $M$
is universal for all functions $b\in\FF_4(a)^n$
with $\|b\|_\infty\le\beta$, whereas
the number $\delta_0=\delta(b)$ may depend on $b$.
\medskip\\
(iii)
We will now fix a pair of initial values $(u_0,h_0)\in
W^{2-2/p}_p(\dot\R^{n+1})\times W^{3-2/p}_p(\R^n)$ satisfying
\eqref{compatibility} and \eqref{sm} with
\begin{equation}
\label{eta} \eta:= 1/(16M_0M),
\end{equation}
where the constants $M_0$ and $M$ are given
in \eqref{Frechet} and  \eqref{FP-5-uniform},
respectively.
We choose
\begin{equation}
\label{b}
b(t):=e^{-D_nt}\gamma v_0,\quad t\ge 0.
\end{equation}
Then
$b\in\FF_4(a)^n$ and $\|b\|_{BC([0,a];BC(\R^n))}\le \|\gamma v_0\|_{BC(\R^n)}
\le \beta$ for any $a>0$,
as $\{e^{-D_nt}: t\ge 0\}$ is a contraction semigroup on $BU\!C(\R^n)$.
Hence the estimate \eqref{FP-5-uniform} holds true
for this (and any other choice) of initial values.
It should be pointed out once more that
the bound $M$ is universal for all initial values
$u_0$ with $\|v_0\|_\infty\le \beta$ - and hence
for $b(t):=e^{-D_nt}\gamma v_0$ -
whereas the number $\delta_0$ may depend on $\gamma v_0$.

We note in passing that $g^\ast_h=0$ for this particular choice of
the function $b$.
Without loss of generality we can assume that $M_0, M\ge 1$.
We shall show that $L_{b}^{-1}K_{b}$ is a contraction on a properly defined subset of
${_0}\EE(a)$ for $a\in (0,\delta_0]$ chosen sufficiently small. For
$r>0$ and $a\in (0,\delta_0]$ we set
\begin{equation*}
{_0}\BB_{\EE(a)}(z^\ast,r):=\{z\in \EE_1(a): z-z^\ast\in
{_0}\EE(a), \ \|z-z^\ast\|_{\EE(a)}< r\}.
\end{equation*}
We remark that $a$ and $r$ are independent parameters that can be
chosen as we please. Let then $r_0>0$ be fixed. It is not
difficult to see that there exists a number $R_0=R_0(u_0,h_0,\delta_0,r_0)$ such
that
\begin{equation*}
\begin{split}
\|\nabla (h+h^*)\|_{BC(J;BC^1)}&+\|h+h^*\|_{\EE_4(a)}
+\|u+u^*\|_{BC(J;BC)} \\
&+Q\big(\|\nabla (h+h^*)\|_{BC(J;BC^1)},\|h+h^*\|_{\EE_4(a)}\big) \le R_0
\end{split}
\end{equation*}
for all   $u\in {_0}\BB_{\EE_1(a)}(0,r)$ and $h\in
{_0}\BB_{\EE_4(a)}(0,r)$, with $a\in (0,\delta_0]$ and $r\in (0,r_0]$
arbitrary, where $z^*=(u^\ast,\pi^\ast,q^\ast,h^\ast)$ is the
solution of equation~\eqref{FP-3}
and where $Q$ is defined in Proposition~\ref{pro:estimates-K}.
Let $M_1:=M_0(1+R_0)$.
It then follows from
Proposition~\ref{pro:estimates-K} and \eqref{FP-5-uniform} that
\begin{equation}
\begin{split}
\label{FP-6-B}
\|D(L_{b}^{-1}K_{b})&(z)\|_{{_0}\EE(a)} \\
&\le M_1M\big[\|b-\gamma(v+v^\ast)\|_{BC(J;BC)\,\cap\,\FF_4(a)}+\|z+z^\ast\|_{\EE(a)}\big] \\
& + M_0M\big[P\big(\|\nabla(h+h^\ast)\|_{BC(J;BC)}\big)\|\nabla(h+h^\ast)\|_{BC(J;BC)}\big] \\
\end{split}
\end{equation}
for all $z\in {_0}\BB_{\EE(a)}(0,r)$ and $a\in (0,\delta_0]$.
\smallskip\\
(iv) For $(u_0,h_0)$
fixed, the norm of $z^*$ in $\EE(a)$ (which involves various
integral expressions evaluated over the interval $(0,a)$) can be
made as small as we like by choosing $a\in (0,\delta_0]$ small. Let
then $a_1\in (0,\delta_0]$ be fixed so that
\begin{equation}
\label{FP-7}
\begin{split}
&\|\nabla h^\ast\|_{BC([0,a_1],BC)}\le 2\eta, \\
& M_1M\big(\|b-\gamma v^\ast\|_{BC([0,a_1];BC)\,\cap\,\FF_4(a_1)}
  +\|z^\ast\|_{\EE(a_1)}\big)\le 1/8.
\end{split}
\end{equation}
Since
$(\nabla h^*,b-\gamma v^*)\in {_0}BC([0,\delta_0],BC(\R^n))$
and $\|\nabla h^*(0)\|_\infty=\|\nabla h_0\|_\infty\le \eta $,
the estimates in \eqref{FP-7}
certainly hold for $a_1$ sufficiently small.
\\
In a next step we choose $2r_1\in (0,r_0]$ so that
\begin{equation}
\label{FP-8}
\begin{split}
& \|\nabla h\|_{{_0}BC([0,a_1],BC(\R^n)} \le \eta, \\
&M_1M\big(\|\gamma v\|_{{_0}BC([0,a_1];BC)\,\cap\,{_0}\FF_4(a_1)}
  +\|z\|_{{_0}\EE(a_1)}\big)\le 1/8,
\end{split}
\end{equation}
for all $h\in {_0}\BB_{\EE(a_1)}(0,2r_1)$, $v\in {_0}\BB_{\EE(a_1)}(0,2r_1)$,
and $z\in {_0}\BB_{\EE(a_1)}(0,2r_1)$.
It follows from Proposition~\ref{pro:6.1} that \eqref{FP-8} can
indeed be achieved.
Combining \eqref{eta}--\eqref{FP-8}  gives
\begin{equation}
\label{FP-9} \|D(L_{b}^{-1}K_{b})(z)\|_{{_0}\EE(a)}\le 1/2, \quad z\in
{_0}\BB_{\EE(a_1)}(0,2r_1)
\end{equation}
showing that
$
L_{b}^{-1}K_{b}:{_0}{\overline\BB}_{\EE(a_1)}(0,r_1)\to {_0}\EE(a_1)
$
is a contraction, where ${_0}{\overline\BB}_{\EE(a_1)}(0,r_1)$
denotes the closed ball in ${_0}\EE(a_1)$ with center at $0$ and
radius $r_1$.
\smallskip\\
It remains so show that $L_{b}^{-1}K_{b}$ maps
${_0}{\overline\BB}_{\EE(a_1)}(0,r_1)$ into itself. From
\eqref{FP-9} and the mean value theorem follows
\begin{equation*}
\begin{split}
\|L_{b}^{-1}K_{b}(z)\|_{{_0}\EE(a_1)} &\le
\|L_{b}^{-1}K_{b}(z)-L_{b}^{-1}K_{b}(0)\|_{{_0}\EE(a_1)}
+\|L_{b}^{-1}K_{b}(0)\|_{{_0}\EE(a_1)}\\
&\le r_1/2+ \|L_{b}^{-1}K_{b}(0)\|_{{_0}\EE(a_1)}, \quad z\in
{_0}\overline{\BB}_{\EE(a_1)}(0,r_1).
\end{split}
\end{equation*}
Here we observe that the norm  of
$L_{b}^{-1}K_{b}(0)=L_{b}^{-1}(K(z^\ast)-(0,f^*_d,g^*,g^*_h))$ in
$_{0}\EE(a_1)$ can be made as small as we wish by choosing $a_1$
small enough. We may assume that $a_1$ was already chosen so that
$\|L_{b}^{-1}K_{b}(0)\|_{{_0}\EE(a_1)}\le r_1/2$.
\smallskip\\
(v) We have shown in (iv) that the mapping
\begin{equation*}
L_{b}^{-1}K_{b}: {_0}\overline{\BB}_{\EE(a_1)}(0,r_1) \to
{_0}\overline{\BB}_{\EE(a_1)}(0,r_1)
\end{equation*}
is a contraction. By the contraction mapping theorem $L_{b}^{-1}K_{b}$
has a unique fixed point $\hat z \in
{_0}\overline{\BB}_{\EE(a_1)}(0,r_1) \subset {_0}\EE(a_1) $ and it
follows immediately from \eqref{FP-1}--\eqref{FP-2} that
$\hat z+z^*$ is the (unique) solution of the nonlinear
problem~\eqref{tfbns2} in
${_0}\overline{\BB}_{\EE(a_1)}(z^*,r_1)$. Setting $t_0=a_1$ gives
the assertion in part (a) of the Theorem.
\smallskip\\
(vi)
The proof that $(u,\pi,q,h)$ is analytic in space and
time proceeds exactly in the same way as in steps (vi)--(vii)
of the proof of Theorem 6.3 in \cite{PrSi09a},
with the only difference that here  $g^\ast_h=g^\ast_{h,\lambda,\nu}=0$,
and that the operator $D_\nu$ in formula (6.30) of \cite{PrSi09a}
is to be replaced by $D_{\lambda,\nu}$, defined by
\begin{equation}
\label{D-lambda-nu}
{\mathcal D}_{\lambda,\nu}h:=(\lambda b_{\lambda,\nu}-\nu|\nabla h),
\qquad b_{\lambda,\nu}(t,x):=b(\lambda t,x+t\nu).
\end{equation}
We note that $D_{1,0}=(b|\nabla\cdot)$.
In the same way as in \cite[Lemma 8.2]{PSS07} one obtains that
\begin{equation}
\label{b-lambda-nu}
[(\lambda,\nu)\mapsto b_{\lambda,\nu}]: (1-\delta,
1+\delta)\times\R^n \to \FF_4(a)
\end{equation}
is real analytic.
The remaining arguments are now the same as in
\cite{PrSi09a},
and this completes the proof of Theorem~\ref{th:nonlinear}.
\end{proof}
\noindent
{\bf Proof of Theorem 1.1:}
Clearly, the compatibility conditions of Theorem~\ref{th:1.1} are satisfied
if and only if \eqref{compatibility} is satisfied.
Moreover, the smallness-boundedness condition of Theorem~\ref{th:1.1}
is equivalent to~\eqref{sm}, where we have slightly abused  notation
by using the same symbol for $u_0$ and its transformed version
$\Theta^\ast_{h_0}u_0$.

Theorem~\ref{th:nonlinear}
yields a unique solution
$(v,w,\pi,[\pi],h)\in \EE(t_0)$ which satisfies the
additional regularity properties listed in part (b) of the theorem.
Setting
\begin{equation*}
(u,q)(t,x,y)=(v,w,\pi)(t,x,y-h(t,x)),\quad (t,x,y)\in{\mathcal O},
\end{equation*}
we then conclude that
$(u,q)\in C^\omega({\mathcal O},\R^{n+2})$ and $[q]\in C^\omega(\mathcal M)$.
The regularity properties listed in Remark~\ref{rem:1.2}(a)
are implied by Proposition~\ref{pro:6.1}(a),(c).
Finally, since $\pi(t,x,y)$ is defined
for every $(t,x,y)\in{\mathcal O}$, we can conclude that
$$q(t,\cdot)\in \dot H^{1}_p(\Omega(t))\subset U\!C(\Omega(t))$$
for every $t\in (0,t_0)$.
\hfill{$\square$}
\bigskip
\goodbreak
\section{Appendix}
In this section we state and prove some technical results
that were used above.
\begin{prop}
\label{pro:6.1} Suppose $p>n+3$.
Then the following embeddings  hold:
\begin{itemize}
\item[(a)] $\EE_1(a)\hookrightarrow
BC(J;W^{2-2/p}_p(\dot\R^{n+1},\R^{n+1})) \hookrightarrow
BC(J;BC^1(\dot\R^{n+1},\R^{n+1}))$ and there is a constant $C_0=C_0(p)$
such that
\begin{equation*}
\|u\|_{{_0}BC(J;W^{2-2/p}_p)}+\|u\|_{{_0}BC(J;BC^1)}\le
C_0\|u\|_{{_0}\EE_1(a)}
\end{equation*}
for all $u\in {_0}\EE_1(a)$ and all $a\in (0,\infty)$. \vspace{1mm}
\item[(b)] $\EE_3(a)\hookrightarrow BC(J;BC(\R^n))$ and there
exists a constant $C_0=C_0(p)$ such that
\begin{equation*}
\|g\|_{{_0}BC(J;BC)}\le C_0  \|g\|_{{_0}\EE_3(a)}
\end{equation*}
for all $g\in {_0}\EE_3(a)$ and all $a\in (0,\infty)$.
\vspace{1mm}
\item[(c)] $\FF_4(a)\hookrightarrow
BC(J;W^{1}_p(\R^n))\cap BC(J;BC^1(\R^n))$ and there
exists a constant $C_0=C_0(p)$ such that
\begin{equation*}
\|g\|_{{_0}BC(J;W^1_p)}+\|g\|_{{_0}BC(J;BC^1)}\le C_0\|g\|_{{_0}\FF_4(a)}
\end{equation*}
for all $g\in{_0}\FF_4(a)$ and all $a\in (0,\infty)$. \vspace{1mm}
\item[(d)] $ \EE_4(a) \hookrightarrow BC^1(J;BC^1(\R^n))\cap
BC(J;BC^2(\R^n)) $ and there exists a constant $C_0=C_0(p)$ such that
\begin{equation*}
\|h\|_{{_0}BC^1(J;BC^1)}+\|h\|_{{_0}BC(J;BC^2)} \le
C_0\|h\|_{{_0}\EE_4(a)}
\end{equation*}
for all $h\in {{_0}\EE_4(a)}$ and all $a\in (0,\infty)$.
\vspace{1mm}
\item[(e)]
$\partial_j\in\Li(\EE_4(a),\EE_3(a))\cap\Li(\EE_4(a),\FF_4(a))$
for $j=1,\cdots,n$.
Moreover, for every given $a_0>0$ there is a constant $C_0=C_0(a_0,p)$
such that
\begin{equation*}
\|\partial_j h\|_{\EE_3(a)}+\|\partial_j h\|_{\FF_4(a)}
\le C_0 \|h\|_{\EE_4(a)}
\end{equation*}
for all $h\in \EE_4(a)$ and all $a\in (0,a_0]$.
\end{itemize}
\end{prop}
\begin{proof}
We refer to \cite[Proposition 6.2]{PSS07} for a proof of (a)-(b).
The assertion in (c) can established in the same way,
using that $\FF_4(a)\hookrightarrow BC(J;W^{2-3/p}_p(\R^n))$,
see \cite[Remark 5.3(d)]{EPS03}.
In order to
show that the embedding constant in (d) does not depend on $a\in
(0,a_0]$ we define an extension operator in the following way: for
$h\in {_0}BC^1([0,a];X)$, with $X$ an arbitrary Banach space, we
first set $\tilde h(t):=0$ for $t\le 0$, so that $\tilde h\in
BC^1((-\infty,a];X)$, and then define
\begin{equation}
\label{E-J}
        (\mathcal Eh)(t):=
        \left\{\begin{array}{ll}
              h(t)   & \mbox{if}\quad 0\le t\le a,\\
              &\vspace{-3mm} \\
              3\tilde h(2a-t)-2\tilde h(3a-2t)& \mbox{if}\quad a\le t. \\
              \end{array}
                  \right.
\end{equation}
A moment of reflection shows that $\mathcal E
h\in{_0}BC^{1}([0,\infty);X)$, and that $\mathcal Eh$ is an
extension of $h$. It is evident that the norm of $\mathcal E:
{_0}BC^1([0,a];X)\to {_0}BC^1([0,\infty);X)$ is independent of
$a\in [0,a_0]$. The assertion follows now by the same arguments as
in the proof of \cite[Proposition 6.2]{PSS07}.
\smallskip\\
Let $a_0>0$ be fixed.
In order to establish part (e) it suffices to show
that there is a constant $C_0=C(a_0,p,r)$ such that
\begin{equation}
\label{H1-Wr}
\|g\|_{W^r_p([0,a];X)}\le C_0\|g\|_{H^1_p([0,a];X)},
\quad a\in (0,a],
\end{equation}
where $X$ is an arbitrary Banach space and $r\in [0,1]$.
This follows from Hardy's inequality as follows:
for $r\in (0,1)$ fixed we have
\begin{equation*}
\begin{split}
\frac{1}{2}\langle g\rangle^p_{W^r_p([0,a];X)}
&=\int_0^a\int_s^a\frac{\|g(t)-g(s)\|^p_X}{(t-s)^{1+rp}}\,dt\,ds\\
&=\int_0^a\int_0^{a-s}\frac{\|g(s+\tau)-g(s)\|^p_X}{\tau^{1+rp}}\,d\tau\,ds\\
&\le \int_0^a\int_0^{a-s}\frac{1}{\tau^{1+rp}}
\left(\int_0^\tau \|\partial g(s+\sigma)\|_X\,d\sigma\right)^p\,d\tau\,ds\\
&\le c(r,p)\int_0^a\int_0^{a-s}\frac{1}{\tau^{1-(1-r)p}}
\|\partial g(s+\tau)\|^p_X\,d\tau\,ds\\
&= c(r,p)\int_0^a \frac{1}{\tau^{1-(1-r)p}}\int_0^{a-\tau}
\|\partial g(s+\tau)\|^p_X\,ds\,d\tau\\
&\le c(r,p)\int_0^a \frac{1}{\tau^{1-(1-r)p}}\,d\tau\;
\|\partial g\|^p_{L_p([0,a];X)}
\end{split}
\end{equation*}
where $\partial g$ is the derivative of $g$,
and this readily yields \eqref{H1-Wr}.
\end{proof}
Our next result will be important in order to derive estimates for
the nonlinearities in~\eqref{tfbns2}.
\begin{lem}
\label{le:2} Suppose $p>n+3$. Let $a_0\in (0,\infty)$ be given.
Then
\begin{itemize}
\item[(a)] $\EE_3(a)$ is a multiplication algebra and we have the
following estimate
\begin{equation}
\label{algebra} \|g_1g_2\|_{\EE_3(a)} \le
(\|g_1\|_\infty+\|g_1\|_{\EE_3(a)})
(\|g_2\|_\infty+\|g_2\|_{\EE_3(a)})
\end{equation}
for all $(g_1,g_2)\in\EE_3(a)\times\EE_3(a)$ and all $a>0$.
\vspace{1mm}
\item[(b)] There exists a constant $C_0=C_0(a_0,p)$ such that
\begin{equation}
\label{M1}
\begin{split}
\|g_1g_2\|_{{_0}\EE_3(a)} &\le C_0
(\|g_1\|_\infty+\|g_1\|_{\EE_3(a)})\| g_2\|_{{_0}\EE_3(a)}
\end{split}
\end{equation}
for all $ (g_1,g_2)\in\EE_3(a)\times{_0}\EE_3(a)$
and all $a\in (0,a_0]$. \vspace{1mm}
\item[(c)] There exists a constant $C_0=C_0(a_0,p)$ such that
\begin{equation}
\label{M2}
\|  g\partial_j h\|_{{_0}\EE_3(a)}
\le C_0\|g\|_{\EE_3(a)}\| h\|_{{_0}\EE_4(a)},\quad j=1,\cdots,n,
\end{equation}
for all $(g,h)\in \EE_3(a)\times {_0}\EE_4(a)$
and $a\in (0,a_0]$. \vspace{1mm}
\item[(d)] Suppose $(g,\psi)\in \EE_3(a)\times\EE_3(a)$ and let
$\beta(t,x):=\sqrt{1+\psi^2(t,x)}$. Then
$\displaystyle\frac{g}{\beta^k}\in\EE_3(a)$ for $k\in\N$ and the
following estimate holds
\begin{equation}
\label{quotient} \left\|\frac{g}{\beta^k}\right\|_{\EE_3(a)} \le
(1+\|\psi\|_{\EE_3(a)})^k(\|g\|_\infty+\|g\|_{\EE_3(a)}).
\end{equation}
\end{itemize}
\end{lem}
\begin{proof}
The assertions in (a)-(b) follow from (the proof of) Proposition
6.6.(ii) and (iv) in \cite{PSS07}.
\medskip\\
(c) To economize our notation we set $r=1/2-1/2p$ and
$\theta=1-1/p$.
\smallskip\\
Suppose that $(g,h)\in \EE_3(a)\times {_0}\EE_4(a)$.
We first observe that
\begin{equation*}
\begin{split}
\|g\partial_j h\|_{W^r_p(J;L_p)}
&\le \big(\|g\|_{L_p(J;L_p)}+\langle g\rangle_{W^r_p(J;L_p}\big)
\|\partial_j h\|_{{_0}BC(J;L_\infty)}
\\
&\quad +\left(\int_0^a\int_0^a
\|g(s)(\partial_jh(t)-\partial_jh(s))\|^p_{L_p}\;
\frac{dt\,ds}{|t-s|^{1+rp}}\right)^{1/p}.
\end{split}
\end{equation*}
Using H\"older's inequality,
and the fact that $(1-r-1/p)=r>0$, we obtain the estimate
\begin{equation}
\label{E3-A}
\begin{split}
&\int_0^a\int_0^a
\|g(s)(\partial_jh(t)-\partial_jh(s))\|^p_{L_p}\;
\frac{dt\,ds}{|t-s|^{1+rp}} \\
&\le \int_0^a\int_0^a
\|g(s)\|^p_{L_p}
\left(\left|\int_s^t\|\partial_t\partial_jh(\tau)\|_{L_\infty}d\tau\right|\right)^p\;
\frac{dt\,ds}{|t-s|^{1+rp}}\\
&\le
\int_0^a \|g(s)\|^p_{L_p}\left(\int_0^a \frac{dt}{|t-s|^{1-(1-r-1/p)p}}\right) ds
\int_0^a \|\partial_t\partial_j h(\tau)|^p_{L_p}d\tau
\\
&\le
C_0(a_0,p)\|g\|^p_{L_p(J;L_p)}\|\partial_t\partial_j h\|^p_{L_p(J;L_\infty)}
\end{split}
\end{equation}
for $a\in (0,a_0]$. Hence we conclude that
\begin{equation}
\label{E3-B}
\begin{split}
\|g\partial_j h\|_{{_0}W^r_p(J;L_p)}
&\le C_0\|g\|_{W^r_p(J;L_p)}
\big(\|\partial_j h\|_{{_0}BC(J;L_\infty)}+\|\partial_t\partial_j h\|_{L_p(J;L_\infty)}\big)\\
&\le C_0
\|g\|_{\EE_3(a)}\|h\|_{{_0}\EE_4(a)}
\end{split}
\end{equation}
uniformly in $a\in (0,a_0]$. It is easy  to verify that
\begin{equation}
\label{E3-C}
\begin{split}
\|g\partial_j h\|_{L_p(J;W^\theta_p)}
&\le \|g\|_{L_p(J;W^\theta_p)}\|\partial_j h\|_{{_0}BC(J;L_\infty)}
+\|g\|_{L_p(J;L_\infty)}\|\partial_j h\|_{{_0}BC(J;W^\theta_p)} \\
&\le C_0 \|g\|_{\EE_3(a)}\|h\|_{{_0}\EE_4(a)}.
\end{split}
\end{equation}
Combining the estimates \eqref{E3-B}--\eqref{E3-C} yields \eqref{M2}.
\medskip\\
(d) As in the proof of Proposition 6.6.(v) in \cite{PSS07} we
obtain
\begin{equation*}
\begin{split}
 \|{g}/{\beta}\|_{W^r_p(J;L_p)}
&\le \|{1}/{\beta}\|_\infty(\| g\|_{L_p(J;L_p)}+\langle
g\rangle_{W^r_p(J;L_p)})
+\| g\|_\infty \langle 1/\beta\rangle_{W^r_p(J;L_p)} \\
&\le (1+\langle 1/\beta\rangle_{W^r_p(J;L_p)}) (\| g\|_\infty+\|
g\|_{W^r_p(J;L_p)}).
\end{split}
\end{equation*}
Thus it remains to estimate the term $ \langle
{1}/{\beta}\rangle_{W^r_p(J;L_p)}$. Using that
$\beta^2(t,x)-\beta^2(s,x)=\psi^2(t,x)-\psi^2(s,x)$ one easily
verifies that
\begin{equation*}
\begin{split}
\left|\frac{1}{\beta(s,x)}-\frac{1}{\beta(t,x)}\right| =
\left|\frac{\beta^2(t,x)-\beta^2(s,x)}{\beta(s,x)\beta(t,x)(\beta(t,x)+\beta(s,x))}\right|
\le |\psi(t,x)-\psi(s,x)|
\end{split}
\end{equation*}
and this yields $\langle 1/\beta \rangle_{W^r_p(J;L_p)} \le
\langle \psi \rangle_{W^r_p(J;L_p)}$. Consequently,
\begin{equation*}
\begin{split}
\|{g}/{\beta}\|_{W^r_p(J;L_p)} &\le (1+\|\psi\|_{W^r_p(J:L_p)})
(\|g\|_\infty+\|g\|_{W^r_p(J;L_p)}).
\end{split}
\end{equation*}
A similar argument shows that
\begin{equation*}
\begin{split}
\|{g}/{\beta}\|_{L_p(J;W^\theta_p)} &\le
(1+\|\psi\|_{L_p(J;W^\theta_p)})
(\|g\|_\infty+\|g\|_{L_p(J;W^\theta_p)}).
\end{split}
\end{equation*}
Combining the last two estimates gives \eqref{quotient} for $k=1$.
The general case then follows by induction.
\end{proof}
\goodbreak
\begin{cor}
\label{cor-E-3}
Suppose $p>n+3$. Let $a_0\in (0,\infty)$
and  $k\in\N$ with $k\ge 1$ be given.
\begin{itemize}
\item[(a)]
There exists a constant $C_0=C_0(a_0,p,k)$ such that
\begin{equation*}
\|(g_1\cdots g_k)\ba g\|_{{_0}\EE_3(a)}
\le C_0\prod_{i=1}^k\big(\|g_i\|_\infty+\|g_i\|_{\EE_3(a)}\big)
\|\ba g\|_{{_0}\EE_3(a)}
\end{equation*}
for all functions $g_i\in\EE_3(a)$, $1\le i\le k$, $\ba g\in {_0}\EE_3(a)$,
and all $a\in (0,a_0]$.
\vspace{2mm}
\item[(b)]
There exists a constant $C_0=C_0(a_0,p,k)$ such that
\begin{equation*}
\label{E3-multi}
\begin{split}
\hspace{1cm}&\|g (\partial_{\ell_1}h\cdots \partial_{\ell_k}h\partial_j\ba h)\|_{{_0}\EE_3(a)} \\
&\le C_0\big(\|\nabla h\|^k_\infty \!+\!
\|h\|_{\EE_4(a)}\|\nabla h\|^{k-1}_\infty
+\|\nabla h\|^k_{BC(J;W^{1-1/p}_p)}\big)\|g\|_{\EE_3(a)}\|\ba h\|_{{_0}\EE_4(a)} \\
&\le C_0\big(\|\nabla h\|^k_{BC(J;BC^1)}+
\|h\|_{\EE_4(a)}\|\nabla h\|^{k-1}_\infty\big)
\|g\|_{\EE_3(a)}\|\ba h\|_{{_0}\EE_4(a)} \\
\end{split}
\end{equation*}
for  $h\in\EE_4(a)$, $\ba h\in{_0}\EE_4(a)$,
$a\in (0,a_0]$, $1\le j\le n$, and $\ell_i\in \{1,\cdots n\}$ with $i=1,\cdots,k$.
\end{itemize}
\end{cor}
\begin{proof}
(a) follows from \eqref{M1} by iteration.
\smallskip\\
(b) The first line in~\eqref{E3-B} shows that
\begin{equation*}
\begin{split}
&\|g (\partial_{\ell_1}h\cdots \partial_{\ell_k}h\partial_j
\ba h)\|_{W^r_p(J;L_p)} \\
&\le C_0\|g\|_{W^r_p(J;L_p)}(\|\partial_{\ell_1}h\cdots \partial_{\ell_k}h
\partial_j\ba h\|_{{_0}BC(J;L_\infty)} +
\|\partial_t(\partial_{\ell_1}h\cdots \partial_{\ell_k}h
\partial_j\ba h)\|_{L_p(J;L_\infty)}).
\end{split}
\end{equation*}
Next we note that
\begin{equation*}
\|\partial_{\ell_1}h\cdots\partial_{\ell_k}h
(\partial_t\partial_j\ba h)\|_{L_p(J;L_\infty)}
\le \|\nabla h\|^{k}_\infty\|\partial_t\partial_j\ba h\|_{L_p(J;L_\infty)},
\end{equation*}
and
\begin{equation*}
\|\partial_{\ell_1}h\cdots(\partial_t\partial_{\ell_i}h)\cdots \partial_{\ell_k}h
\partial_j\ba h\|_{L_p(J;L_\infty)}
\!\le \|\partial_t\partial_{\ell_i}h\|_{L_p(J;L_\infty)}
\|\nabla h\|^{k-1}_\infty \|\partial_j\ba h\|_{{_0}BC(J;L_\infty)}.
\end{equation*}
Proposition 6.1(d) now implies the assertion
for $\|\cdot\|_{W^r_p(J;L_p)}$.
On the other hand we have by \eqref{E3-C}
for $\theta=1-1/p$
\begin{equation*}
\begin{split}
&\|g (\partial_{\ell_1}h\cdots \partial_{\ell_k}h\partial_j
\ba h)\|_{L_p(J;W^\theta_p)} \\
&\le \|g\|_{L_p(J;W^\theta_p)}
\|\partial_{\ell_1}h\cdots \partial_{\ell_k}h\partial_j
\ba h\|_{\infty}
+\|g\|_{L_p(J;L_\infty)}\|\partial_{\ell_1}h\cdots \partial_{\ell_k}h\partial_j
\ba h\|_{{_0}BC(J;W^\theta_p)} \\
&\le
C_0\|g\|_{\EE_3(a)}\big(\|\nabla h\|^k_\infty
+\|\nabla h\|^k_{BC(J;W^{1-1/p}_p)}\big)\|h\|_{{_0}\EE_4(a)}
\end{split}
\end{equation*}
since $W^\theta_p(\R^n)$ is a multiplication algebra.
The last inequality then follows from Sobolov's embedding theorem.
\end{proof}
\begin{rem}
\label{remark-E-3}
It can be shown that the estimate in~\eqref{M2} can be
improved as follows:
For every $a_0\in (0,\infty)$ there is a constant
$C_0=C_0(a_0,p)>0$ and a constant $\theta=\theta(p)>0$ such that
\begin{equation*}
\label{M2-b}
\|  g\partial_j h\|_{{_0}\EE_3(a)}
\le C_0\, a^\theta\|g\|_{\EE_3(a)}\| h\|_{{_0}\EE_4(a)}
\end{equation*}
holds for all $(g,h)\in \EE_3(a)\times {_0}\EE_4(a)$
and $a\in (0,a_0]$.
In the same way, the constant $C_0$ in
Corollary~\ref{cor-E-3}(b) can be replaced by $C_0 a^\theta$.
\end{rem}
\begin{lem}
\label{le:3} Suppose $p>n+3$. Let $a_0\in (0,\infty)$ be given.
Then
\begin{itemize}
\item[(a)] $\FF_4(a)$ is a multiplication algebra and we have the
estimate
\begin{equation*}
\label{F-algebra}
\begin{split}
\|g_1g_2\|_{\FF_4(a)}
&\le C_a\|g_1\|_{\FF_4(a)}\|g_2\|_{\FF_4(a)}
\end{split}
\end{equation*}
for all $(g_1,g_2)\in\FF_4(a)\times\FF_4(a)$,
where the constant $C_a$ depends on $a$.
 \vspace{1mm}
\item[(b)] There exists a constant $C_0=C_0(a_0,p)$ such that
\begin{equation}
\label{F-M1}
\begin{split}
\|g_1g_2\|_{{_0}\FF_4(a)} &\le C_0
(\|g_1\|_\infty+\|g_1\|_{\FF_4(a)})\| g_2\|_{{_0}\FF_4(a)}
\end{split}
\end{equation}
for all $ (g_1,g_2)\in\FF_4(a)\times{_0}\FF_4(a)$
and all $a\in (0,a_0]$. \vspace{1mm}
\item[(c)] There exists a constant $C_0=C_0(a_0,p)$ such that
\begin{equation}
\label{F-M2}
\|  g\partial_j h\|_{{_0}\FF_4(a)} \le
C_0 (\|g\|_\infty+\|g\|_{\FF_4(a)})\| h\|_{{_0}\EE_4(a)},\quad j=1,\cdots,n,
\end{equation}
for all $(g,h)\in \FF_4(a)\times {_0}\EE_4(a)$
and $a\in (0,a_0]$. \vspace{1mm}
\end{itemize}
\end{lem}
\begin{proof}
 Here we equip $\FF_4(a)$ with the (equivalent) norm
 \begin{equation}
\label{norm}
\|g\|_{\FF_4(a)}=\|g\|_{W^{1-1/2p}_p(J;L_p)}+
\sum_{i=1}^n\|\partial_i g\|_{L_p(J;W^{1-1/p}_p)}\;.
\end{equation}
(a) This follows from Proposition ~\ref{pro:6.1}(c)
by similar arguments as in the proof of
Proposition 6.6(ii) and (iv) in \cite{PSS07}.
\smallskip\\
(b) It follows from part (a) and Proposition~\ref{pro:6.1}(c) that
\begin{equation*}
\|g_1g_2\|_{{_0}W^r_p(J;L_p)}
\le C_0 (\|g_1\|_\infty + \|g_1\|_{W^r_p(J;L_p)})\|g_2\|_{{_0}\FF_4(a)},
\quad (g_1,g_2)\in \FF_4(a)\times {_0}\FF_4(a)
\end{equation*}
where $r=1-1/2p$.
Next, observe that again by Proposition~\ref{pro:6.1}(c)
\begin{equation*}
\begin{split}
\|(\partial_i g_1)g_2\|_{L_p(J;W^{\theta}_p)}
&\le \|\partial_i g_1\|_{L_p(J;W^{\theta}_p)}\|g_2\|_{{_0}BC(J;L_\infty)}
+\|\partial_i g_1\|_{L_p(J;L_\infty)}\|g_2\|_{{_0}BC(J;W^{\theta}_p)}\\&\le C_0\|g_1\|_{\FF_4(a)}\|g_2\|_{{_0}\FF_4(a)}
\end{split}
\end{equation*}
where $\theta=1-1/p$.
Moreover,
\begin{equation*}
\begin{split}
\|g_1\partial_i g_2\|_{L_p(J;W^{\theta}_p)}
&\le \|g_1\|_{L_p(J;W^{\theta}_p)}\|\partial_i g_2\|_{{_0}BC(J;L_\infty)}
+\|g_1\|_\infty \|\partial_i g_2\|_{L_p(J;W^{\theta}_p)}\\
& \le C_0(\|g_1\|_\infty+\|g_1\|_{\FF_4(a)})\|g_2\|_{{_0}\FF_4(a)}.
\end{split}
\end{equation*}
The estimates above in conjunction with \eqref{norm}
yields \eqref{F-M1}.
\smallskip\\
(c) follows from (b) by setting $g_2=\partial_j h$
and from Proposition~\ref{pro:6.1}(e),
which certainly is also true for ${_0}\EE_4(a)$.
\end{proof}
\smallskip
\noindent
{\bf Proof of Proposition~\ref{pro:estimates-K}:}

It follows as in the proof of \cite[Proposition 6.2]{PrSi09a}
that $N_b\in C^\omega(\EE(a),\FF(a))$, and moreover,
that
$DH_b(z)\in \Li({_0}\EE(a),{_0}\FF(a))$ for $z\in\EE(a)$.
It thus remains to prove the estimates
stated in the proposition.

Without always writing this explicitly, all the
estimates derived below will be uniform in $a\in (0,a_0]$, for
$a_0>0$ fixed. Moreover, all estimates will be uniform for
$(\ba u,\ba \pi,\ba q,\ba h)\in {_0}\EE(a)$.
\goodbreak
\smallskip
\noindent
(i)
Without changing notation we consider here
the extension of $h$
from $\R^n$ to $\R^{n+1}$ defined by
$h(t,x,y)=h(t,x)$ for $(x,y)\in\R^n\times\R$ and $t\in J$.
With this interpretation we have
\begin{equation}
\label{F-1}
\begin{split}
\|\partial  h\|_{\infty, J\times \R^{n+1}}
=\|\partial  h\|_{\infty, J\times \R^{n}},
\quad h\in\EE(a),
\quad \partial\in\{\partial_j,\Delta,\partial_t\},
\end{split}
\end{equation}
where $\|\cdot\|_{\infty,U}$
denotes the sup-norm for the set $U\subset J\times \R^{n+1}$.
Next we observe that
\begin{equation}
\label{F-2}
\begin{split}
& BC(J;BC(\R^{n+1}))\cdot L_p(J;L_p(\R^{n+1}))
\hookrightarrow L_p(J;L_p(\R^{n+1})),\\
& BC(J;L_p(\R^{n+1}))\cdot L_p(J;BC(\R^{n+1}))
 \hookrightarrow L_p(J;L_p(\R^{n+1})),\\
& BC(J;BC(\R^{n+1}))\cdot BC(J;BC(\R^{n+1}))
\hookrightarrow BC(J;BC(\R^{n+1})),
\end{split}
\end{equation}
that is, multiplication is continuous and bilinear in the
indicated function spaces (with norm  equal to 1).
\smallskip\\
Let us first consider the term $F_1(u,h):=|\nabla h|^2\partial^2_y
u$ appearing in the definition of~$F$. Its  Fr\'echet
derivative at $(u,h)$ is given by
\begin{equation*}
DF_1(u,h)[\ba u,\ba h] =|\nabla h|^2\partial^2_y\ba u +2
 (\nabla h|\nabla \ba h)\partial^2_y u.
\end{equation*}
Suppose $(\ba u,\ba h)\in {_0}\EE_1(a)\times{_0}\EE_4(a)$. From
\eqref{F-1}, the first and third line in \eqref{F-2} and
Proposition~\ref{pro:6.1}(d) follows the estimate
\begin{equation*}
\begin{split}
\|DF_1(u,h)[\ba u,\ba h]\|_{{_0}\FF(a)}
\le C_0\|\nabla h\|_\infty(\|\nabla h\|_{\infty}+ \|u\|_{\EE_1(a)})
(\|\ba u\|_{{_0}\EE_1(a)}+\|\ba h\|_{{_0}\EE_4(a)})\\
\end{split}
\end{equation*}
for all $(u,h)\in\EE_1(a)\times \EE_4(a)$. It is important to
note that the constant $C_0$ does not depend on the
length of the interval $J=(0,a)$ for $a\in (0,a_0]$.

Next, let us take a closer look at the term
$F_2(u,h):=\Delta h\partial_y u$ in the definition of $F$.
The Fr\'echet derivative is
$
DF_2(u,h)[\ba u,\ba h]=\Delta h\partial_y\ba u
+\Delta\ba h \partial_y u.
$
We infer from \eqref{F-1}, the first and second line in
\eqref{F-2}, and Proposition~\ref{pro:6.1} that
\begin{equation*}
\begin{split}
\|DF_2(u,h)[\ba u,\ba h]\|_{{_0}\FF(a)} &\le
(\|\Delta h\|_{L_p(J;L_\infty) }+\|\partial_yu\|_{L_p(J;L_p)})\cdot\\
&\hspace{1cm} (\|\partial_y\ba u\|_{_{0}BC(J;L_p)}
+\|\Delta\ba h\|_{{_0}BC(J;L_\infty)})\\
&\le C_0(\|h\|_{\EE_4(a)}+\|u\|_{\EE_1(a)})
(\|\ba u\|_{{_0}\EE_1(a)}+\|\ba h\|_{{_0}\EE_4(a)})\\
\end{split}
\end{equation*}
for all $(u,h)\in \EE_1(a)\times\EE_4(a)$.
\medskip\\
The derivative of $F_3(u,h):=(u| \nabla h)\partial_y u$, where
$\nabla h:=(\nabla h,0)$, is given by
\begin{equation*}
DF_3(u,h)[\ba u,\ba h]= (\ba u|\nabla h)\partial_y u +
(u|\nabla h)\partial_y\ba u + (u| \nabla\ba h) \partial_y u
\end{equation*}
and it follows from \eqref{F-1}--\eqref{F-2} and
Proposition~\ref{pro:6.1}(a),(d) that there
is a constant $C_0>0$ such that
\begin{equation*}
\begin{split}
\|DF_3(u,h)[\ba u,\ba h]\|_{{_0}\FF(a)}
\le C_0 (\|\nabla h\|_\infty +\|u\|_\infty)\|u\|_{\EE_1(a)}
(\|\ba u\|_{{_0}\EE_1(a)}+\|\ba h\|_{{_0}\EE_4(a)})\\
\end{split}
\end{equation*}
for all $(u,h)\in \EE_1(a)\times\EE_4(a)$.
\smallskip\\
Let us also consider the term $F_4(u,h):=\partial_t h\partial_yu$.
Observing that
\begin{equation*}
DF_4(u,h)[\ba u,\ba h] =\partial_t h\partial_y\ba u+\partial_t\ba
h\partial_y u,
\end{equation*}
that
$\partial_t:\EE_4(a)\to \FF_4(a)$
is linear and continuous and
\begin{equation}
\label{F-5}
\begin{split}
\FF_4(a)\hookrightarrow L_p(J; BC^1(\R^n))\cap BC(J; BC^1(\R^n))
\end{split}
\end{equation}
we conclude from \eqref{F-1}--\eqref{F-5} and Proposition
\ref{pro:6.1}(a),(c) that there is a constant $C_0=C_0(a_0)$ such
that
\begin{equation*}
\begin{split}
\|DF_4(u,h)[\ba u,\ba h]\|_{{_0}\FF(a)}
&\le (\|\partial_th\|_{L_p(J;L_\infty)}+\|\partial_y u\|_{L_p(J;L_p)})\\
&\hspace{1cm}
        (\|\partial_y\ba u\|_{{_0}BC(J;L_p)}+
        \|\partial_t\ba h\|_{{_0}BC(J;L_\infty)})\\
&\le C_0(\|h\|_{\EE_4(a)}+\|u\|_{\EE_1(a)})
(\|\ba u\|_{{_0}\EE_1(a)}+\|\ba h\|_{{_0}\EE_4(a)})\\
\end{split}
\end{equation*}
for all $(u,h)\in \EE_1(a)\times\EE_4(a)$.

The derivative of $F_5(\pi,h):=\partial_y\pi\nabla h$
is given by
$$
DF_5(\pi,h)[\ba\pi,\ba h] =\partial_y\ba\pi\nabla h+\partial_y\pi\nabla\ba h.
$$
It follows from \eqref{F-1}--\eqref{F-2} and
Proposition~\ref{pro:6.1}(d) that there exists $C_0$ such
that
\begin{equation*}
\begin{split}
\|DF_5(\pi,h)[\ba \pi,\ba h]\|_{{_0}\FF(a)}
&\le (\|\nabla h\|_\infty+\|\partial_y\pi\|_{L_p(J;L_p)})\\
&\hspace{1cm}
        (\|\partial_y \ba \pi\|_{L_p(J;L_p)}+\|\nabla h\|_{{_0}BC(J;L_\infty)})\\
&\le C_0(\|\nabla h\|_\infty+\|\pi\|_{\EE_2(a)})
(\|\ba\pi\|_{{_0}\EE_2(a)}+\|\ba h\|_{{_0}\EE_4(a)})\\
\end{split}
\end{equation*}
for all $(\pi,h)\in \EE_2(a)\times \EE_4(a)$. The remaining terms
in the definition of $F$ can be analyzed in the same way.
Summarizing we have shown that there is a constant $C_0$
such that
\begin{equation}
\label{F1-summary}
\begin{split}
\|DF(u,\pi,h)&[\ba u,\ba\pi,\ba h]\|_{{_0}\FF_1(a)} \\
\le
& C_0\big[\|\nabla h\|_\infty +\|\nabla h\|^2_{\infty}+\|(u,\pi,h)\|_{\EE(a)} \\
&{\hskip1.7cm}+(\|\nabla h\|_{\infty}\!+\|u\|_{\infty})
\|u\|_{\EE_1(a)}\big]
\|(\ba u,\ba\pi,\ba h)\|_{{_0}\EE(a)}
\end{split}
\end{equation}
for all $(u,\pi,h)\in \EE(a)$ and all $a\in (0,a_0]$.
\smallskip\\
(ii) We will now  consider the nonlinear function
$F_d(u,h)=(\nabla h| \partial_y v)$, stemming from the
transformed divergence. Since $h(x,y):=h(x)$ does not depend on
$y$ we have
\begin{equation}
\label{Fd-1} F_d(u,h)=(\nabla h| \partial_yu)
=\partial_y (\nabla h | u).
\end{equation}
We note that
\begin{equation}
\label{Fd-2}
\partial_y\in {\Li}
\big(H^1_p(J;L_p(\R^{n+1})),
H^1_p(J;H^{-1}_p(\R^{n+1}))\big).
\end{equation}
The norm of this linear mapping does not depend on the length of
the interval $J=[0,a]$. It is easy to see that multiplication is
continuous in the following function spaces:
\begin{equation}
\label{Fd-3}
\begin{split}
&H^1_p(J;BC(\R^{n+1}))\cdot H^1_p(J;L_p(\R^{n+1}))
\hookrightarrow H^1_p(J;L_p(\R^{n+1})) \\
&BC(J;BC^1(\dot\R^{n+1}))\cdot L_p(J;H^1_p(\dot\R^{n+1}))
\hookrightarrow L_p(J;H^1_p(\dot\R^{n+1})).
\end{split}
\end{equation}
The derivative of $F_d$ at
$(u,h)\in \EE_1(a)\times\EE_4(a)$ is given by
\begin{equation*}
DF_d(u,h)[\ba u,\ba h] =(\nabla h|\partial_y\ba  u)
+ (\nabla\ba h |\partial_y u) =\partial_y ((\nabla h|\ba u)+
(\nabla\ba h| u)).
\end{equation*}
We want to derive
a uniform estimate for $DF_d(u,h)[\ba u,\ba h]$
which does not depend on the length of the interval $J=[0,a]$.
We conclude from \eqref{F-1}--\eqref{F-5} that
\begin{equation*}
\begin{split}
&\| (\nabla h|\ba u ) \|_{{_0}H^1_p(J;L_p)}
\sim\|(\nabla h|\ba u)\|_{L_p(J;L_p)} +\|(\partial_t\nabla h |\ba u)\|_{L_p(J;L_p)}
+\|(\nabla h|\partial_t\ba u)\|_{L_p(J;L_p)}\\
&\le \|\nabla  h\|_\infty\|\ba u\|_{L_p(J;L_p)}
+\|\partial_t\nabla h\|_{L_p(J;L_\infty)}\|\ba u\|_{{_0}BC(J;
L_p)}
+\|\nabla h\|_\infty \|\partial_t\ba u\|_{L_p(J;L_p)}\\
&\le C_0 (\|\nabla h\|_\infty +\|h\|_{\EE_4(a)})\|\ba
u\|_{{_0}H^1_p(J;L_p)}.
\end{split}
\end{equation*}
Similar arguments also yield
$
\| (\nabla\ba h| u )\|_{{_0}H^1_p(J;L_p)} \le C_0
\|u\|_{H^1_p(J;L_p)}\|\ba h\|_{{_0}\EE_4(a)}.
$
These estimates in combination with \eqref{Fd-2} show that there
is a constant $C_0$ such that
\begin{equation*}
\begin{split}
&\| (\nabla h|\partial_y\ba u)+(\partial_y u|\nabla\ba h)\|_{{_0}H^1_p(J;H^{-1}_p)} \\
&\quad \le C_0(\|\nabla h\|_\infty +\|h\|_{\EE_4(a)}+\|u\|_{\EE_1(a)})
 (\|\ba u\|_{{_0}\EE_1(a)}+\|\ba h\|_{{_0}\EE_4(a)})
\end{split}
\end{equation*}
for all $(u,h)\in \EE_1(a)\times \EE_4(a)$, where  $C_0$ is
uniform in $a\in (0,a_0]$. Observing that
\begin{equation*}
\|(\nabla h|\partial_y\ba u )\|_{L_p(J;L_p)} +\Sigma_{j=1}^{n+1}
\|(\partial_j\nabla h|\partial_y\ba u)
+(\nabla h | \partial_j\partial_y \ba u) \|_{L_p(J;L_p)}
\end{equation*}
defines an equivalent norm for
$\|(\nabla h |\partial_y\ba u )\|_{L_p(J;H^1_p)}$, we infer once more from
\eqref{F-1}--\eqref{F-2} and Propostion~\ref{pro:6.1} that
\begin{equation*}
\begin{split}
&\|(\nabla h|\partial_y\ba u )\|_{L_p(J;H^1_p)}
\le C_0(\|h\|_{\EE_4(a)}+\|\nabla h\|_\infty)\|\ba u\|_{{_0}\EE_1(a)}\\
&\|(\nabla\ba h|\partial_y u)\|_{L_p(J;H^1_p)} \le
C_0\|u\|_{L_p(J;H^2_p)}\|\ba h\|_{{_0}\EE_4(a)}.
\end{split}
\end{equation*}
Summarizing we have shown that there exists a constant $C_0$ such
that
\begin{equation}
\label{Fd-summary}
\begin{split}
\|DF_d(u,h)&[\ba u,\ba h]\|_{{_0}\FF_2(a)} \\
&\le C_0 (\|\nabla h\|_\infty+\|h\|_{\EE_4(a)}+\|u\|_{\EE_1(a)})
(\|\ba u\|_{{_0}\EE_1(a)}+\|\ba h\|_{{_0}\EE_4(a)})
\end{split}
\end{equation}
for all $(u,h)\in \EE_1(a)\times \EE_4(a)$ and $a\in (0,a_0]$.
\medskip\\
\noindent (iii) We remind that
\begin{equation}
\label{jump} [\![\mu \partial _i\;\cdot]\!]
\in{\Li}\big(H^1_p(J;L_p(\dot\R^{n+1}))\cap
L_p(J;H^2_p(\dot\R^{n+1})),\EE_3(a)\big)
\end{equation}
where $[\![\mu \partial _i u]\!]$ denotes the jump of the quantity
$\mu\partial _i u$ with $u$ a generic function from
$\dot\R^{n+1}\to\R$, and where $\partial _i =\partial_{x _i} $ for
$i=1,\ldots,n$ and $\partial_{n+1} =\partial_y$.
\smallskip\\
The mapping $G(u,q,h)$ is made up of terms of the form
\begin{equation*}
[\![\mu \partial_i u_k]\!] \partial_j h, \quad [\![\mu\partial_i
u_k]\!] \partial_j h\partial_l h, \quad q\partial_j h, \quad
\Delta h\partial_j h, \quad G_\kappa(h), \quad
G_\kappa(h)\partial_jh
\end{equation*}
where $u_k$ denotes the $k$-th component of a function
$u\in\EE_1(a)$. It follows from Lemma~\ref{le:2}(a) and
\eqref{jump} that the mappings
\begin{equation*}
\begin{split}
&(u,h)\mapsto [\![\mu \partial_i u_k]\!] \partial_j h,\
[\![\mu\partial_i u_k]\!] \partial_j h\partial_l h
:\EE_1(a)\times\EE_4(a)\to\EE_3(a), \\
& (q,h)\mapsto q\partial_jh :\EE_3(a)\times\EE_4(a)\to \EE_3(a),
\quad h\mapsto \Delta h\partial_jh:\EE_4(a)\to \EE_3(a)
\end{split}
\end{equation*}
are multilinear and continuous.
Let us now take a closer look at the term $G_1(u,h):=[\![\mu
\partial_i u_k]\!] \partial_j h$. Its Fr\'echet derivative is
given by
\begin{equation*}
DG_1(u,h)[\ba u,\ba h]=\partial_j h [\![\mu \partial_i \ba u_k]\!]
+[\![\mu \partial_i u_k]\!] \partial_j\ba h.
\end{equation*}
Setting $g_1=\partial_j h$ and $g_2:=[\![\mu \partial_i \ba
u_k]\!]$ we obtain from \eqref{M1} and \eqref{jump} the estimate
\begin{equation*}
\|\partial_j h [\![\mu \partial_i \ba u_k]\!]\|_{{_0}\EE_3(a)} \le
C_0 (\|\nabla h\|_\infty + \|\nabla h\|_{\EE_3(a)}) \|\ba
u\|_{{_0}\EE_1(a)}.
\end{equation*}
On the other hand, setting $g:=[\![\mu \partial_i u_k]\!]$ we
conclude from \eqref{M2} and \eqref{jump} that
\begin{equation*}
\|[\![\mu \partial_i u_k]\!] \partial_j\ba h\|_{{_0}\EE_3(a)} \le
C_0 \|u\|_{\EE_1(a)}\|\ba h\|_{{_0}\EE_4(a)}.
\end{equation*}
Consequently,
\begin{equation}
\label{G1-estimate}
\begin{split}
\|DG_1(u,h)&[\ba u,\ba h]\|_{{_0}\EE_3(a)} \\
&\le C_0 (\|\nabla h\|_\infty +\!\|\nabla
h\|_{\EE_3(a)}+\|u\|_{\EE_1(a)}) (\|\ba u\|_{{_0}\EE_1(a)}+ \|\ba
h\|_{{_0}\EE_4(a)})
\end{split}
\end{equation}
for all $(u,h)\in\EE_1(a)\times \EE_4(a)$, and all $a\in (0,a_0]$.
\medskip\\
Given $(u,h)\in\EE_1(a)\times\EE_4(a)$ let
$G_2(u,h):=[\![\mu\partial_i u_k]\!] \partial_j h\partial_l h$.
The Fr\'echet derivative of $G_2$ is given by
\begin{equation*}
DG_2(u,h)[\ba u, \ba h] = \partial_j h\partial_l h
[\![\mu\partial_i\ba u_k]\!] +[\![\mu\partial_i u_k]\!]\partial_j
h \partial_j\ba h +[\![\mu\partial_i u_k]\!]\partial_l h
\partial_j\ba h.
\end{equation*}
From Corollary \ref{cor-E-3}(a),(b) and \eqref{jump} follows that
there is a constant $C_0$ such that
\begin{equation}
\label{G2-estimate}
\begin{split}
\|DG_2(u,h)[\ba u,\ba h]\|_{{_0}\EE_3(a)}
&\le C_0 (\|\nabla h\|_\infty +\|h\|_{\EE_4(a)})^2\|\ba u\|_{{_0}\EE_1(a)} \\
&\!\!\!\! + C_0\big(\|\nabla h\|_{BC(J;BC^1)}
+\|h\|_{\EE_4(a)}\big)\|u\|_{\EE_1(a)}\|\ba h\|_{{_0}\EE_4(a)})
\end{split}
\end{equation}
for all $(u,h)\in\EE_1(a)\times \EE_4(a)$
and all $a\in (0,a_0]$.
\medskip\\
The terms $G_3(q,h):=q\partial_j h$ and $G_4(h):=\Delta
h\partial_jh$ can be analyzed in the same way as the term $G_1$,
yielding the following estimates
\begin{equation}
\label{G3-estimate}
\begin{split}
\|DG_3(q,h)&[\ba q,\ba h]\|_{{_0}\EE_3(a)} \\
&\le C_0 (\|\nabla h\|_\infty \! +\|\nabla
h\|_{\EE_3(a)}\!+\|q\|_{\EE_3(a)}) (\|\ba q\|_{{_0}\EE_3(a)}+ \|\ba
h\|_{{_0}\EE_4(a)})
\end{split}
\end{equation}
as well as
\begin{equation}
\label{G4-estimate}
\begin{split}
\|DG_4(h)&\ba h\|_{{_0}\EE_3(a)} \le C_0 (\|\nabla h\|_\infty+\|\nabla
h\|_{\EE_3(a)}+\|\nabla^2 h\|_{\EE_3(a)}) \|\ba
h\|_{{_0}\EE_4(a)}.
\end{split}
\end{equation}
\smallskip
Let us now consider the term
\begin{equation*}
G_5(h)=\frac{1}{(1+\beta)\beta}(\partial_jh)^2\Delta h, \quad
\beta(t,x):=\sqrt{1+|\nabla h(t,x)|^2},
\end{equation*}
appearing in the definition of $G_\kappa $. The Fr\'echet
derivative of $G_5$ at $h$ is given by
\begin{equation*}
\label{G5-derivative}
\begin{split}
DG_5(h)\ba h&= -\Big(\frac{1}{(1\!+\!\beta)^2\beta^2}
+\frac{1}{(1\!+\!\beta)\beta^3}\Big)
(\partial_j h)^2\Delta h\partial_kh\partial_k\ba h \\
&\hspace{0.5cm}+ \frac{1}{(1\!+\!\beta)\beta}\big(2\partial_jh\Delta
h\partial_j\ba h+(\partial_j h)^2\Delta\ba h\big).
\end{split}
\end{equation*}
Before continuing, we note that the term $1/(1+\beta)$ can be
treated in exactly the same way as $1/\beta$, as a short
inspection of the proof of Lemma~\ref{le:2}(d) shows. It follows
then from Corollary~\ref{cor-E-3}(a)--(b) and from \eqref{quotient}
that there is a constant $C_0$ such that
\begin{equation}
\label{G5-estimate}
\|DG_{5}(h)\ba h\|_{{_0}\EE_3(a)}
\le
C_0\big[P(\|\nabla h\|_\infty)\!+\! Q(\|\nabla h\|_{BC(J;BC^1)},\|h\|_{\EE_4(a)})\big]\|\ba h\|_{{_0}\EE_4(a)}
\end{equation}
for all $h\in\EE_4(a)$ and all $a\in (0,a]$,
where $P$ and $Q$ are polynomials
with coefficients equal to one and vanishing zero-order terms.
Analogous arguments can be used for the remaining terms
${(\nabla h | \nabla^2 h\nabla h)}/{\beta^3}$ and
$G_\kappa(h)\partial_j h$ appearing in $G$, yielding the same
estimate as in \eqref{G5-estimate}.
\smallskip\\
(iv) It remains to consider the nonlinear term
$H_b(v,h):=(b-\gamma v| \nabla h)$.
The Fr\'echet derivative is given by
$
DH_b(v,h)[\ba v,\ba h]=-(\nabla h| \gamma\ba v) +(b-\gamma v |\nabla\ba h).
$
From Lemma~\ref{le:3}(b) with $g_1=\partial_j h$ and
$g_2=\gamma\ba v_k$, where $\ba v_k$ denotes the $k$-th component of $\ba v$, follows
$
\|(\nabla h| \gamma\ba v)\|_{{_0}\FF_4(a)}
\le C_0(\|\nabla h\|_\infty+\|h\|_{\EE_4(a)})\|\ba v\|_{{_0}\EE_1(a)}.
$
Lemma~\ref{le:3}(c) with $g=(b-\gamma v)_k$ and $h=\ba h$
implies
\begin{equation*}
\begin{split}
\|(b-\gamma v|\nabla\ba h)\|_{{_0}\FF_4(a)}
\le C_0 (\|b-\gamma v\|_\infty +\|b-\gamma v\|_{\FF_4(a)})
 \|\ba h\|_{{_0}\EE_4(a)}.
\end{split}
\end{equation*}
We have, thus, shown that
\begin{equation}
\label{G-b}
\begin{split}
\|DH_b(v,h)\| \le C_0(\|\nabla h\|_{\infty} + \|h\|_{\EE_4(a)}
+\|b-\gamma v\|_{BC(J;BC)\cap \FF_4(a)}).
 \end{split}
\end{equation}
Combining the estimates in
\eqref{F1-summary}, \eqref{Fd-summary} and
\eqref{G1-estimate}--\eqref{G-b}
yields the assertions of
Proposition~\ref{pro:estimates-K}.
\hfill{$\square$}

\end{document}